%% file: Saccon-Singular-Bifurcation.tex
\documentclass[11pt, twoside, a4paper]{article}

\usepackage{amsmath}     				
\usepackage{amssymb}   					
\usepackage{amsthm}     				
\usepackage{graphicx}    				

\usepackage{textcomp}    				
\usepackage[T1]{fontenc} 				
\usepackage{marvosym}    				
\usepackage[sc]{mathpazo}  	 			

\usepackage{authblk} 					
\usepackage[usenames]{xcolor}  			
\usepackage{lastpage} 					

\usepackage{enumerate}	 				
\usepackage{mathtools}					
\usepackage[noadjust]{cite}			

\usepackage[width=136mm]{caption}		

\definecolor{ForestGreen}{rgb}{0.15,0.416,0.18}
\definecolor{EgyptBlue}{rgb}{0.063,0.2,0.65}

\usepackage[english]{babel}
\usepackage{csquotes,xpatch}

\usepackage{tikz-cd} 

\usepackage{times}

\usepackage{hyperref} 					
\hypersetup{
	colorlinks=true,
	linkcolor=EgyptBlue,         
   citecolor=EgyptBlue,          
   urlcolor=ForestGreen          
}

\title{Radial Solutions and a Local Bifurcation Result  for a Singular Elliptic Problem with Neumann Condition}
\author{
  \qquad
  Claudio Saccon
   \footnote{Dipartimento di Matematica,
     Largo Pontecorvo 5, I56127 Pisa, ITALY,
     e-mail: claudio.saccon@unipi.it}
  }
\date{}

\input{defs.tex}

\newcommand{\hh}{\ensuremath{\mathcal{H}}}

\newcommand{\HHH}{\ensuremath{\mathbb{H}}}

\newcommand{\LLL}{\ensuremath{\mathbb{L}}}

\newcommand{\doi}[1]{\url{https://doi.org/#1}}

\newcommand{\ZB}[1]{\href{https://zbmath.org/?q=an:#1}{Zbl~#1}}

\begin{document}
 \maketitle

\begin{abstract}
 We study the problem $-\Delta u=\lm u-u^{-1}$ with a Neumann boundary condition; the peculiarity being the
 presence of the singular term $-u^{-1}$. We point out that the minus sign in front of the negative power of $u$ is particularly challenging, since no convexity argument can be invoked. Using bifurcation techniques we are able to prove the existence of solution $(u_\lm,\lm)$ with $u_\lm$ approaching the trivial constant solution $u=\lm^{-1/2}$ and $\lm$ close to an eigenvalue of a suitable linearized problem. To achieve this we also need to prove a generalization of a classical two-branch bifurcation result for potential operators. Next  we study the radial case and show that in this case one of the bifurcation branches is global and we find the asymptotical behavior of such a branch. This results allows to derive the existence of multiple solutions
 $u$ with $\lm$ fixed.
\end{abstract}

{\bf MSC:}\quad
34B16,35J20,58E07

{\bf Keywords:}\quad
Singular elliptic equation; Positive solutions; Radial solutions; Variational bifurcation; Two branches; Eigenvalue problem. 

%
\section{Introduction}

In the last decades several authors have studied semilinear elliptic 
problems with singular nonlinear term (with respect to the unknown function $u$). The model problem is the following:
\begin{equation}\label{eqn:singular-problems-gen}
 \begin{cases}
  -\Delta u= \gamma u^{-q}+f(x,u)&\mbox{in }\Omega,
  \\
  u>0&\mbox{in }\Omega,
  \\
  u=0&\mbox{on }\partial\Omega,
 \end{cases}\end{equation}
where $q>0$ $\gamma\neq0$ and $f$ is a non linear term with standard 
growth conditions. Existence and multiplicity of solutions to problem \eqref{eqn:singular-problems-gen} are usually investigated in terms of
the behavior of $f$ and the sign of $\gamma$.

A main aspect to be taken into account is the variational nature of
\eqref{eqn:singular-problems-gen}: formally speaking solutions $u$ of
\eqref{eqn:singular-problems-gen} are expected to be 
critical points of the functional:
\begin{equation*}
 I(u):=\frac12\int_\Omega|\nabla u|^2\,dx-
 \frac{\gamma}{1-q}\int_\Omega u^{1-q}\,dx-
 \int_\Omega F(x,u)\,dx,
\end{equation*}
defined on $\Hone$ and
restricted to $\set{u\ge0}$,
where $F(x,s)$ is a primitive in $s$ of $f(x,s)$ 
(if $q=1$ a logarithm should be introduced). 
Unfortunately the presence of the singular term
makes it problematic to give a rigorous formulation of the above ideas.

The majority of the known results concern the case $\gamma>0$, where
the term $u\mapsto-\dfrac{\gamma}{1-q}u^{1-q}$ is convex 
in the interval $]0,+\infty[$. 
This fact helps a lot, whether one tries to directly deal with  $I$
(by using some nonsmooth-critical-point theory) or to use an approximation scheme (by a sequence $I_n\to I$, $I_n$
being $\mathcal{C}^1$ on $\Hone$). For instance, if $f=0$, the problem has a unique weak solution $\bar u$ in $\Hone$ when $0<q<3$ and the solution is a minimizer for $I$. This result can be extended for all $q>0$ dropping the request that $\bar u\in\Hone$ (see \cite{CaninoDegiovanni2004singularvariational,CranRabTar1977}). For a small non exaustive list of multiplicity results for solutions of this kind of 
problems see
\cite{AdimurthiGiacomoniSantra2018,BoccardoOrsina2010SemilinearEE,CoclitePalmieri1989,DhanyaEunkyungShivaji2015,GherguRadulescu2003,GherguRadulescu2008BOOK,GiacomoniMukherjeeSreenadh2019,Haitao2003,HiranoSacconShioji2008,Saccon2014SingularNonsymmetric,ShiYao1998} (and the references therein).

If we turn to  $\gamma<0$ the literature is scarcer: to the author's knowledge the main results are contanined in \cite{ChoiLazerMcKenna1998,DiazMorelOswald1987,MontSilva2012,ShiYao1998,Zhang1995}. In this situation solutions are ``attracted'' from the value  zero and tend to develop ``dead cores'', so the formulation \eqref{eqn:singular-problems-gen} needs to be modified in order to admit non strictly positive solutions. For instance  the only solution for the case $f=0$ is $u=0$ (as one can easily see by multiplying the equation by $u$). Moreover a direct variational approach using the functional $I$ seems difficult for the  moment and the usual approach goes by perturbation methods. 

We have found particularly interesting the paper \cite{MontSilva2012} by Montenegro and Silva, where the authors use perturbation methods and  show that there exist two nontrivial solutions when $\gamma=-1$, 
$0<q<1$,  $f(x,u)=\mu u^p$, with $q<p<$ and $\mu>0$ big enough.
If we pass to $q=1$, simple tests in the radial case suggest that the Dirichlet problem only has the trivial solution. As we said  before solutions starting from zero are ``forced to stick'' at zero and not allowed to ``emerge'' (in contrast with the case of $q<1$). On this respect see Remark \eqref{eqn:rmk-no-dirichlet-radial-soln}.

For this reasons, in the case $q=1$,  we are lead to replace  the Dirichelet condition with a Neumann one. In particular we have considered the
problem:

  \begin{equation}\label{eqn:singular-neumann-problem}
  \begin{cases}
  -\Delta u=\lm u-\dfrac{1}{u}&\mbox{ in }\Omega,
   \\
   u>0&\mbox{ in }\Omega,
   \\
   \nabla u\cdot\nu=0&\mbox{ on }\partial\Omega
   \end{cases}
 \end{equation}

where $\Omega$ is a bounded smooth open subset of $\real^N$ and $\nu$ denotes the unit normal defined on $\partial\Omega$. This corresponds to
the problem of \cite{MontSilva2012} with $q=p=1$
(with Neumann condition).

In case $N=1$ \eqref{eqn:singular-neumann-problem} is closely related to a problem studied  by Del Pino, Man\'asevich,  and Montero in 1992
(see \cite{DelPinoManaMonte1992}) who deal with an ODE, 
in the periodic case, with a  more general, non autonomous, singular term $f(u,x)$ (singular in $u$ and $T$-periodic in $x$).
Using topological degree arguments they 
prove for instance that the equation:
\begin{equation*}
 -\ddot u=\lm u-\frac{1}{u^\alpha}
 \qquad,\qquad
 u(x)>0
 \quad,\qquad
 u(x+T)=u(x),
\end{equation*}
where $\alpha\geq1$, has a solution provided 
$\lm\neq\dfrac{\mu_k}{4}$ for all $k$. Here $\mu_k$ denote the eigenvalues of a suitable linearized problem which arises in a natural way from the problem. In this case, which has a variational structure,  the results of 
\cite{DelPinoManaMonte1992} can be derived from the existence of two global bifurcation branches which originate from 
trivial solutions of the linearized problem.

In this paper we present two types of results concerning problem
\eqref{eqn:singular-neumann-problem}. In Theorem \eqref{thm-local-bufurcation} of section 2 we prove the existence of two local bifurcation branches 
$(u_{1,\rho},\lm_{1,\rho})$ and $(u_{2,\rho},\lm_{2,\rho})$ of
solutions of \eqref{eqn:singular-neumann-problem}, such that 
$(u_{i,\rho},\lm_{1,\rho})\to(\hat u,\hat\lm)$, as $\rho\to 0$,
where $\hat\lm/2$ is an eigenvalue of 
$-\Delta$ with Neumann condition and $\hat u$ is the constant function: $\hat u\equiv\hat\lm^{-1/2}$. 
The proof of \eqref{thm-local-bufurcation} heavily relies
on a variant of the
well known abstract results on the existence of two bifurcation branches in the variational case (see \cite{Marino1973Bifur,Boehme1972,JiaquanLiuBifur,RabinowitzBifurPotOp77}). 
To the author surprise such a 
variant (see Theorem \eqref{thm:abstract-bifurcation-theorem} seems not to be present in the literature so its proof is carried on in section 3. It has to be said that proving
\eqref{thm:abstract-bifurcation-theorem} requires some additional
technicalities compared to the standard version. Indeed in \cite{Marino1973Bifur,Boehme1972}) the proof goes by studying a suitable perturbed function $f_\rho$ on the unit sphere $S$, while in our case $S$ has to be replaced by a sphere-like set $S_\rho$ also varying with $\rho$. This requires to construct suitable projections to
show that all $S_\rho$'s are omeomorphic to $S_0=S$ (for $\rho$ small). Apart from this the proof of \eqref{thm:abstract-bifurcation-theorem} follows the ideas of \cite{BenCanBifur}.

In section 4 we study the radial case in dimension $N=2$ (the same
could be probably done for $N\geq3$) using ODE techniques and a continuation argument for the nodal regions of the solutions.
In this way, following the ideas of
\cite{Rabinowitz1971GlobalResults}, we are able to prove that one of the two branches $(u_\rho,\lm_\rho)$ is global and bounded in 
$\lm_\rho$. This is done by proving that nodal regions of $u_\rho$ cannot collapse along the branch and that $\lm_\rho\to\bar\lm$ 
as $\norm{u_\rho}\to+\infty$, where $\bar\rho$ is an eigenvalue
of another suitable linear problem. In this way -- in the radial case -- we can find a lower estimate in the number of solutions for 
a fixed $\lm$, by counting the number of branches that cross
$\lm$.

\section{A local bifurcation result for the singular problem}
Let $\Omega$ be a bounded  open subset of $\real^N$  with smooth boundary.

\begin{thm}\label{thm-local-bufurcation}
 Let $\hat\mu>0$ be an eigenvalue of the following Neumann problem:
 \begin{equation}\label{eqn:eigenvalue-neumann-problem}
  \begin{cases}
   -\Delta u=\mu u&\mbox{ in }\Omega,
   \\
   \nabla u\cdot\nu=0&\mbox{ on }\partial\Omega,
  \end{cases}
 \end{equation}
 ($\nu$ denotes the normal to $\partial\Omega$). 
 
 Then there exists $\rho_0>0$ such that for all $\rho\in]0,\rho_0[$
 there exist two distinct pairs $(u_{1,\rho},\lm_{1,\rho})$ and 
 $(u_{2,\rho},\lm_{2,\rho})$ such that, for $i=1,2$:
 \begin{equation*}
  (u_{i,\rho},\lm_{i,\rho})\mbox{ are solutions of \eqref{eqn:singular-neumann-problem}}
  \ ,\ 
  u_{i,\rho}\to\frac{1}{\sqrt{\mu/2}}\mbox{\quad ( in $\Hon$  ) }
  \ ,\ 
  \lm_{i,\rho}\xrightarrow{\rho\to0}\frac{\hat\mu}{2}.
 \end{equation*}
\end{thm}
\begin{proof}
 We start by introducing some changes of variables.  First of all notice that, for all $\lm>0$, Problem  \eqref{eqn:singular-neumann-problem} has the constant solution $u(x)=\dfrac{1}{\sqrt{\lm}}$. 
 If we seek for solutions of the form
 $u=\dfrac{1}{\sqrt{\lm}}+z$ we easily end up with the equivalent problem on $z$:
%
\begin{equation}\label{eqn:z-neumann-problem}
  \begin{cases}
  -\Delta z=2\lm z-h_\lm(z)&\mbox{ in }\Omega,
   \\
   \sqrt{\lm}z>-1&\mbox{ in }\Omega,
   \\
   \nabla z\cdot\nu=0&\mbox{ on }\partial\Omega
   \end{cases}
 \end{equation}
 where $h_{\lm}:\left]-\dfrac{1}{\sqrt{\lm}},+\infty\right[\to\real$
 is defined by:
 \begin{equation*}
  h_\lm(s)=\frac{\lm\sqrt{\lm}s^2}{1+\sqrt{\lm}s}.
 \end{equation*}
 Now we consider another simple transformation: $v:=\sqrt{\lm}z$,
 so that \eqref{eqn:z-neumann-problem} turns out to be equivalent to:
 \begin{equation}\label{eqn:v-neumann-problem}
  \begin{cases}
  -\Delta v=2\lm\left( v-\dfrac12 h_1(v)\right)&\mbox{ in }\Omega,
   \\
   v>-1&\mbox{ in }\Omega,
   \\
   \nabla v\cdot\nu=0&\mbox{ on }\partial\Omega
   \end{cases}
 \end{equation}
 
 Now choose $s_0$ with $0<s_0<1/2$ and a $\CC^\infty$ cutoff function $\eta:\real\to [0,1]$ such that $\eta(s)=1$ for $|s|\leq s_0$, $\eta(s)=0$, for $|s|\geq2s_0$. Define
 $\tilde h_1:\real\to\real$ by 
 \begin{equation}\label{eqn:def-tilde-h-1}
  \tilde h_1(s):=\eta(s)\,h_1(s)
 \end{equation}
 (we agree that $\tilde h_1(-1)=0$).
 Then $\tilde h_1\in\mathcal{C}^\infty_0(\real;\real)$, 
 $\tilde h_1'(0)=h_1''(0)=0$, $\tilde h_1=h_1$ on $[-s_0,s_0]$.
 Denote by $\tilde H_1:\real\to\real$ 
 the primitive function for $\tilde h_1$ 
 (i.e. $\tilde H_1'=\tilde h_1$) such that $\tilde H_1(0)=0$.
 
 Now we apply the bifurcation theorem 
 \eqref{thm:abstract-bifurcation-theorem}
 with $\HHH:=\Hon$. $\LLL=\Ltwo$, $\hh=0$, $\hat\lm=\mu$,
 $\hh_1(v):=\frac12\int\limits_\Omega\tilde H_1(v)\,dx$.
 In this way
we get that there exists
$\rho_0>0$ such that for all $\rho\in]0,\rho_0[$ there are two 
distinct pairs
$(v_{1,\rho},\mu_{1,\rho})$ and $(v_{2,\rho},\mu_{2,\rho})$ which 
are \emph{weak} solutions of:
 
\begin{equation}\label{eqn:v-tilde-neumann-problem}
  \begin{cases}
  -\Delta v=\mu\left( v-\frac12\tilde h_1(v)\right)&\mbox{ in }\Omega,
   \\
   \nabla v\cdot\nu=0&\mbox{ on }\partial\Omega
   \end{cases}
 \end{equation}
and such that 
\begin{equation}
\label{eqn:infinitesimal}
 v_{i,\rho}\xrightarrow{\rho\to0}0\quad\mbox{( in $\Hon$ )}
 \quad,\quad
 \mu_{i,\rho}\xrightarrow{\rho\to0}\mu_k
 \qquad i=1,2.
\end{equation}
 We claim the there exists a constant $K$ such, for any $\mu\in[\hat\mu_k-1,\hat\mu_k+1]$ and  any weak solution $v$ of \eqref{eqn:v-tilde-neumann-problem}, $v$ is bounded and:
  \begin{equation}
 \label{eqn:Linfty-estimate}
  \|v\|_\infty\leq K\mu\norm{v-\frac12\tilde h_1(v)}_2.
 \end{equation}
 For this we use a standard bootstrap argument using the fact that the function 
 $k(s)=\left(s-\frac12\tilde h_1(s)\right)$, appearing on the right hand side of
 \eqref{eqn:v-tilde-neumann-problem}, verifies 
 \begin{equation}
 \label{eqn:k-lipschitz}
  |k(s)|\leq M|s|\qquad
  \forall s\in\real
 \end{equation}
 for a suitable $M$
 (since $\tilde h_1'$ is bounded).
 Assume that $v$ is a solution, i.e. $-\Delta v=\mu k(v)$,  and $v\in\Lpi[q]$ for some
 $q>1$ (for sure this is true for $q=2^*$). Then, by \eqref{eqn:k-lipschitz}, $k(v)\in\Lpi[q]$. From the standard Calderon--Zygmund theory (see e.g. Section 9.6 in \cite{GilbargTrudingerBook}), we have $v\in W^{2,q}(\Omega)$. Then, using the Sobolev embedding Theorem, either $v\in\Lpi[q_1]$ with $q_1\leq\frac{Nq}{N-2q}$ (if $2q\leq N$)  or
 $v\in\CC^{0,\alpha}$ with $\alpha>0$
 (in the case $2q>N$). Iterating this argument a finite number of times we get the conclusion. Notice that we could go further and prove that
 $v$ is $\CC^\infty$ and is a classical solution.
 
%
%
%
%

 Using  \eqref{eqn:infinitesimal} and \eqref{eqn:Linfty-estimate}
 we get that $v_{i,\rho}\to0$ in $\Linfty$ as $\rho\to0$, so $|v_{i,\rho}|<s_0$, $i=1,2$,
 for $\rho_0$ small. 
 This implies that $\tilde h_1(v_{i,\rho})=h_1(v_{i,\rho})$, and
 $v_{i,\rho}$ actually solve \eqref{eqn:v-neumann-problem} with 
 $\lm_{i,\rho}:=\dfrac{\mu_{i,\rho}}{2}$. \qquad
 Going backwards and
 setting  $u_{i,\rho}:=\dfrac{1}{\sqrt{\lm_{i,\rho}}}+\sqrt{\lm_{i,\rho}}v_{i,\rho}$, we find the desired solutions of \eqref{eqn:singular-neumann-problem}.

\end{proof}

\section{A variant for the two bifurcation branches
theorem for potential operators}
 
Let $\LLL$ and $\HHH$ be two Hilbert spaces  such that $\HHH\subset\LLL$ with
a compact embbedding $i:\HHH\to \LLL$. We use the notations
$\norm{\cdot}$, $\scal{\cdot}{\cdot}$ and $\norm{\cdot}_\LLL$, $\scal{\cdot}{\cdot}_\LLL$  to indicate the norms and inner products in $\HHH$ 
 and $\LLL$ respectively. Let $A:\HHH\to \HHH$ be a bounded linear simmetric operator such that:
 \begin{equation}\label{eqn:coercitivity}
  \scal{Au}{u}\geq\nu\norm{u}^2-M\norm{u}^2_\LLL\qquad\forall u\in \HHH
 \end{equation}
where $\nu>0$ and $M$ are two constants.
We say that $\lm\in\real$ is an ``eigenvalue for $A$'' if there exists
$e\in \HHH\setminus\set{0}$ with
\begin{equation*}
 \scal{Ae}{v}=\lm\scal{e}{v}_\LLL
 \qquad\forall v\in \HHH
\end{equation*}
which corresponds to say that:
\begin{equation*}
 Ae=\lm i^*e.
\end{equation*}
In this case we say that $e$ is an ``eigenvector'' corrsponding to $\lm$.

 It is well kwown that there exists a sequence $(\lm_n)$ of
 eigenvalues of $A$ with
 $\lm_n\leq\lm_{n+1}$, $\lm_n\to+\infty$, such that the corresponding eigenvectors generate $\HHH$. It is convenient to agreee
 that $\lm_0=-\infty$.
 We can suppose that for any $k\geq1$ we are given an eigenvector $e_k$ relative to $\lm_k$ with
 $\norm{e_n}_\LLL=1$, and:
 \begin{equation*}
  \scal{e_n}{e_m}=\scal{e_n}{e_m}_\LLL=0\qquad\mbox{ if }n\neq m.
 \end{equation*}
 If $\lm\in\real$ we define:
 \begin{equation*}
  E^-_\lm:=\spanned\set{e_i\st\lm_i<\lm}
  ,\ 
  E^0_\lm:=\spanned\set{e_i\st\lm_i=\lm}
  ,\ 
  E^+_\lm:=\overline{\spanned\set{e_i\st\lm_i>\lm}}^{(\HHH)}
 \end{equation*}
 ($E^0_\lm=\set{0}$ if $\lm$ is not an eigenvalue). If $\lm_n\leq\lm\leq\lm_{n+1}$ it is clear that:
 \begin{equation*}
  \sup_{\set{u\in E^-_\lm\st\norm{u}_\LLL=1}}\scal{Au}{u}\leq\lm_n
  \quad ,\quad 
  \inf_{\set{u\in E^+\lm\st\norm{u}_\LLL=1}}\scal{Au}{u}\geq\lm_{n+1},
 \end{equation*}
 while $\scal{Au}{u}=\lm$, if $u\in E^0_\lm$.

\begin{thm}[Bifurcation]\label{thm:abstract-bifurcation-theorem}
Let $\hh\in\CC^1(\HHH;\real)$, $\hh_1\in\CC^1(\LLL;\real)$
be such that:
\begin{equation}\label{eqn:ipotesi-su-h-h1}
\begin{gathered}
 \hh(0)=0,\quad 
 \nabla \hh(0)=0,\quad
 \lim_{u\to0}\frac{\norm{\nabla \hh(u)}_\LLL}{\norm{u}_\LLL}=0
 \\
 \hh_1(0)=0,\quad 
 \nabla_\LLL \hh_1(0)=0,\quad
 \lim_{u\to0}\frac{\norm{\nabla_\LLL \hh_1(u)}_\LLL}{\norm{u}_\LLL}=0.
\end{gathered}
\end{equation}
Notice that we are using the symbol $\nabla$ to denote the gradient with
respect to the inner product in $\HHH$ and $\nabla_\LLL$ for the corresponding
gradient in $\LLL$.

 Let $\hat\lm$ be an eigenvalue for $A$.
 Then, for any $\rho>0$ small, there exist $(u_{1,\rho},\lm_{1,\rho})$ and 
 $(u_{2,\rho},\lm_{2,\rho})$ which solve 
 the problem:
 \begin{equation}\label{eqn:abstract-bifurcation-equation}
  A u+\nabla \hh(u)=
  \lm i^*\left(u+\nabla_\LLL \hh_1(u)\right)
  \qquad\qquad
  u\neq0.
 \end{equation}
 such that 
 $u_{1,\rho}\neq u_{2,\rho}$ and:
 \begin{equation}\label{eqn:abstract-bifurcation-limits}
  u_{1,\rho}\xrightarrow{\HHH}0
  \quad,\quad
  u_{2,\rho}\xrightarrow{\HHH}0
  \quad,\quad
  \lm_{1,\rho}\to\hat\lm
  \quad,\quad
  \lm_{2,\rho}\to\hat\lm
  \qquad\qquad\mbox{as }\rho\to0.
 \end{equation}
\end{thm}
\newcommand{\const}{\frac14}
\newcommand{\twoconst}{2\frac12}
\begin{proof}
 We adapt the proof of Lemma 3.4 in \cite{BenCanBifur}.
 Let $\hat\lm=\lm_i=\lm_k$  with $\lm_{i-1}<\lm_i$ and 
 $\lm_k<\lm_{k+1}$.
  We define
$f:\HHH\to\real$ and $g:\LLL\to\real$ by:
\begin{equation*}
 f(u):=\frac{1}{2}\scal{Au}{u}+\hh(u)
 \quad,\quad
 g(u):=\frac{1}{2}\norm{u}_\LLL^2+\hh_1(u),
\end{equation*}

Let $\CC:=\set{u\in\LLL\st1<\norm{u}_\LLL<2}$.
Moreover, if $0<\rho<1$ we define:
\begin{gather*}
 f_\rho(u):=\frac{1}{\rho^2}f(\rho u),
 \quad 
 g_\rho(u):=\frac{1}{\rho^2}g(\rho u),
 \\
 \hh_\rho(u):=\frac{1}{\rho^2}\hh(\rho u),
 \quad
 \hh_{1,\rho}(u):=\frac{1}{\rho^2}\hh_1(\rho u),
 \\
 \SSS_\rho:=\set{u\in\CC\st g_\rho(u)=1}.
\end{gather*}
In fact 
$f_\rho(u)=\dfrac{1}{2}\scal{Au}{u}+\hh_\rho(u)$
and $g_\rho(u)=\dfrac{1}{2}\norm{u}_\LLL^2+\hh_{1,\rho}(u)$.

%
%
%

Since the result we are proving only involves the behaviour of $\hh,\hh_1$ near zero, we are allowed to modify $\hh$ and $\hh_1$ outside of
a small ball around the origin.
More precisely using \eqref{eqn:ipotesi-su-h-h1} we can find $R$ in $]0,1/3[$ such that
\begin{equation}\label{eqn:small-gradients}
 \norm{\nabla\hh(u)}\leq\frac{\nu}{8}\norm{u}\ 
 \forall u\mbox{ with }\norm{u}<3R,
 \quad
 \norm{\nabla_\LLL\hh_1(u)}\leq\frac{1}{2}\norm{u}_1\ 
 \forall u\mbox{ with }\norm{u}_1<3R,
\end{equation}
and define $\tilde\hh(u):=\eta(\norm{u})\hh(u)$, 
$\tilde\hh_1(u):=\eta(\norm{u}_1)\hh_1(u)$, where $\eta:[0,+\infty[\to[0,1]$ is a cutoff function with $\eta(s)=1$ for $0\leq s\leq R$,  $\eta(s)=0$ for $s\geq3R$, and $\eta'(s)\leq1$.  Now since
\[
   \tilde\hh(u)=\hh(u)\quad\forall u\mbox{ with } \norm{u}<R,\qquad
   \tilde\hh_1(u)=\hh_1(u)\quad\forall u\mbox{ with } \norm{u}_\LLL<R,
\]
then the conclusion of Theorem \ref{thm:abstract-bifurcation-theorem} holds for $\hh,\hh_1$ if and only if it holds for $\tilde\hh,\tilde\hh_1$. Indeed the first component $u_\rho$
of a bifucation branch (for any of the two problems) eventually verifies $\norm{u_\rho}<R$ and $\norm{u_\rho}_1<R$. So from now on we
replace $\hh$ with $\tilde\hh$ and $\hh_1$ with $\tilde\hh_1$, maintaining the same notation. With simple computations we can deduce from \eqref{eqn:small-gradients} that the redifined functions verify:
%
%
%
%
\begin{equation}\label{eqn:h-h1-nabla-small}
\thetag{a}\quad
\left|\nabla\hh(u)\right|\leq\frac\nu4\norm{u}
 \ \forall u\in \HHH,
 \qquad
 \thetag{b}\quad
 \left|\nabla_\LLL\hh_1(u)\right|\leq\norm{u}_\LLL
 \ \forall u\in \LLL.
\end{equation}
From \thetag{a} in \eqref{eqn:h-h1-nabla-small} we get 
\begin{equation}\label{eqn:h-quadratic}
 |\hh(u)|\leq\frac\nu4\norm{u}^2
 \Rightarrow
 |\hh_\rho(u)|\leq\frac\nu4\norm{u}^2
 \qquad
 \forall u\in \HHH,\,\forall\rho\in[0,1]
\end{equation}

%
Using \eqref{eqn:coercitivity} and \eqref{eqn:h-quadratic} we get  that:
\begin{equation}
\label{eqn:sublevels-bounded}
\norm{u}^2\leq \frac4\nu\left(f_\rho(u)+M\norm{u}^2_\LLL\right)
\end{equation}

From \eqref{eqn:ipotesi-su-h-h1} and \eqref{eqn:sublevels-bounded}
we  easily get that, if $c\in\real$ and  $\rho\to0$:
\begin{equation}\label{eqn:estimates-h-h1}
  \begin{gathered}
  \sup_{u\in\CC,f_\rho(u)\leq c}|\hh_\rho(u)|\to0\quad,\quad 
  \sup_{u\in\CC}|\hh_{1,\rho}(u)|\to0,
  \\
  \sup_{u\in\CC,f_\rho(u)\leq c}\norm{\nabla \hh_\rho(u)}\to0
  \quad,\quad 
  \sup_{u\in\CC}\norm{\nabla_\LLL \hh_{1,\rho}(u)}_\LLL\to0.
  \end{gathered}
\end{equation}

So if we extend the definition to $\rho=0$ by letting
$f_0(u):=\frac12\scal{Au}{u}$ and $g_0(u):=\frac12\norm{u}^2_\LLL$,
then  $(\rho,u)\mapsto f_\rho(u)$ is continuous on $[0,+\infty[\times \HHH$ 
and $(\rho,u)\mapsto g_\rho(u)$ is continuous on $[0,+\infty[\times \LLL$. 
We also define $\SSS_\rho$ for $\rho=0$:
\begin{equation*}
 \SSS_0:=\set{u\in \LLL\st g_0(u)=1}=\set{u\in \LLL\st \norm{u}_\LLL^2=2}.
\end{equation*}
Notice that the critical values of $f_0$ on $\SSS_\rho$ are precisely the eigenvalues $\lm_n$.

We claim that there exist $\bar\rho>$ such that 
the $\LLL$-closure of $\SSS_\rho$ is contained in $\CC$ for all  $\rho\in]0,\bar\rho]$ in other terms $\SSS_\rho$ is closed for $\rho>0$ small.
Indeed if the claim were false there would exist two sequences $(\rho_n)$ and  
$(u_n)$ such that $\rho_n\to0$, $g_{\rho_n}(u_n)=\frac{\norm{u_n}_\LLL^2}{2}+\frac{\hh_1(\rho_n u_n)}{\rho_definen^2}=1$, and $\norm{u_n}_\LLL\in\set{1,2}$.
From \eqref{eqn:ipotesi-su-h-h1} we would have $\frac{\hh_1(\rho_n u_n)}{\rho_n^2}=\frac{\hh_1(\rho_n u_n)}{\norm{\rho_n u_n}_\LLL^2}\norm{u_n}_\LLL^2\to0$, so $\norm{u_n}_\LLL\to\sqrt{2}$ which yields a contradiction for $n$ large.

Let us split $\HHH$ as $\HHH=\X_1\oplus\X_2\oplus\X_3$, where:
\begin{equation*}
 \X_1:=E_{\hat\lm}^-
 \quad,\quad
 \X_2:=E_{\hat\lm}^0
 \quad,\quad
 \X_3:=E_{\hat\lm}^+.
\end{equation*}
and consider the orthogonal projections $\Pi_i:\HHH\to\X_i$. $i=1,2,3$. We also denote $\Pi_{13}:=\Pi_1+\Pi_3$. 
Given $\rho\in[0,\bar\rho]$ and $\delta\in]0,1[$, we set:
\begin{equation*}
 \CC_\delta:=\set{u\in\CC\st\norm{\Pi_2(u)}_\LLL\geq\delta}
 \quad,\qquad
 \SSS_{\rho,\delta}:=\SSS_\rho\cap\CC_\delta.
\end{equation*}
Since $\SSS_\rho$ is closed, then
$S_{\rho,\delta}$ is a smooth manifold with boundary, the boundary being:
\begin{equation*}
 \Sigma_{\rho,\delta}:=
 \set{u\in\SSS_\rho\st\norm{\Pi_2(u)}_\LLL=\delta}.
\end{equation*}
Notice that $\SSS_{0,\delta}\neq\emptyset$ (
$\delta<1$). Let us indicate by $\bar f_{\rho}$ the restriction
of $f_\rho$ on $\SSS_{\rho,\delta}$. 

We will use the notion of
lower critical point for $\bar f_\rho$ (see \cite{BenCanBifur, MarSacSNS1997} and the references therein):
$u$ is (lower) critical for $\bar f_\rho$ if and only there exist
$\lm,\mu\in\real$ such that $\mu\geq0$, $\mu=0$ if 
$\norm{\Pi_2(u)}_\LLL>\delta$,  and:
\begin{equation}\label{eqn:u-critico-su-Cdelta}
 \scal{Au}{v}+\scal{\nabla \hh_\rho(u)}{v}=
 \lm\scal{u+\nabla_\LLL \hh_{1,\rho}(u)}{v}_\LLL+\mu\scal{\Pi_2(u)}{v}_\LLL
 \quad\forall v\in \HHH.
\end{equation}

Define  $\Gamma:\CC_\delta\times[1/2,2]\to\CC_\delta$ and 
$\ph:[0,\bar\rho]\times\CC_\delta\times[1/2,2]\to\real$ by:
\begin{align*}
 \Gamma(u,t):=&
 \frac{\delta\Pi_2(u)}{\norm{\Pi_2(u)}_\LLL}+
 t\left(u-\frac{\delta\Pi_2(u)}{\norm{\Pi_2(u)}_\LLL}\right)=
 \\
 &t\Pi_{13}(u)+
 (\delta+t\left(\norm{\Pi_2(u)}_\LLL-\delta\right))
  \frac{\Pi_2(u)}{\norm{\Pi_2(u)}_\LLL}
 \\
 \ph(\rho,u,t):=&g_{\rho}\left(\Gamma(u,t)\right)
\end{align*}
%
%
%
%
%
%
%
%
With easy computations:
\begin{multline*}
\ph(0,u,t)=
\frac{1}{2}
\left(
 t^2\norm{\Pi_{13}(u)}_\LLL^2+
 \left(
 \delta+t\left(\norm{\Pi_2(u)}_\LLL-\delta\right)
 \right)^2
\right)=
\\
 \frac12\left(
 t^2(\norm{u}^2_\LLL-2\delta\norm{\Pi_2(u)}_\LLL+\delta^2)+
 2\delta t(\norm{\Pi_2(u)}_\LLL-\delta)+\delta^2
 \right)
\end{multline*}
Since $1<\norm{u}_\LLL<2$ and $\norm{\Pi_2(u)}_\LLL\geq\delta$, we have:
\begin{equation*}
\frac{t^2}{2}-
2\delta t^2
\leq
 \ph(0,u,t)\leq
 \left(2+\frac{\delta^2}{2}\right)t^2+2\delta t+\frac{\delta^2}{2}.
\end{equation*}
In particular:
\begin{equation*}
 2-2\delta\leq\ph(0,u,2)
 \quad,\quad
 \ph\left(0,u,1/2\right)\leq
 \frac{1}{2}+\frac{\delta^2}{8}+\delta+\frac{\delta^2}{2}<\frac{1}{2}+2\delta.
\end{equation*}
We can choose $\delta_0>0$ so that 
$2-2\delta>3/2$ and 
$\dfrac{1}{2}+2\delta<3/4$
for all $\delta\in]0,\delta_0]$. 
From now on we consider   $0<\delta\leq\delta_0$. By \eqref{eqn:estimates-h-h1}, up to shrinking $\bar\rho$, 
we have:
\begin{equation*}
 \sup_{u\in\CC_{\delta}}\ph(\rho,u,1/2)<1,
 \qquad
 \inf_{u\in\CC_{\delta}}\ph(\rho,u,2)>1
 \qquad\qquad
 \forall\rho\in[0,\bar\rho].
\end{equation*}
Moreover:
\begin{equation*}
  \frac{\partial}{\partial t}\ph(0,u,t)=
 t(\norm{u}^2_\LLL-2\delta\norm{\Pi_2(u)}_\LLL+\delta^2)+
 \delta(\norm{\Pi_2(u)}_\LLL-\delta)\geq
 t(1-4\delta)
\end{equation*}
so, up to shrinking $\delta_0$, we have
 $\dfrac{\partial}{\partial t}\ph(0,u,t)\geq\frac{1}{4}$
 for all $t\geq\dfrac{1}{2}$.  Up to further shrinking $\bar\rho>0$ (again we use \eqref{eqn:estimates-h-h1}),
 we have that $\rho\in[0,\bar\rho]$, $\delta\in]0,\delta_0]$,
 $u\in\CC_{\delta}$ imply
 \begin{equation*}
  \ph(\rho,u,1/2)<1,\quad
  \ph(\rho,u,2)>1,\quad 
  \frac{\partial}{\partial t}\ph(\rho,u,t)\geq\frac{1}{8}
  \quad
  \forall t\in[1/2,2].
 \end{equation*}
 We can therefore conclude that for all $u\in[0,\bar\rho]$ and $u\in\CC_\delta$ there exists a unique $\bar t=\bar t(\rho,u)$ in
 $[1/2,2]$ such that $\ph(\rho,u,\bar t(u,\rho))=1$, that is
 $\Gamma(u,\bar t(\rho,u))\in\SSS_{\rho,\delta}$. It is easy to 
 check that $\bar t:[0,\bar\rho]\times\CC_\delta\to[1/2,2]$ is continuous
 and so is  $\Phi:[0,\bar\rho]\times\CC_\delta\to\SSS_{\rho,\delta}$ defined by 
 $\Phi(\rho,u):=\Gamma(u,\bar t(\rho,u))$.
 Notice that:
 \begin{equation*}
  t\in[1/2,2], u\in\CC_\delta,\ \norm{\Pi_2(u)}_\LLL=\delta
  \ \Rightarrow\ 
   \norm{\Pi_2(\Gamma(u,t))}_\LLL=\delta.
 \end{equation*}
 Therefore
 $\Phi(\rho,\cdot)$ maps 
 $\set{u\in\CC_\delta,\norm{\Pi_2(u)}_\LLL=\delta}$
 into $\Sigma_{\rho,\delta}$. Also notice that 
 $\Phi(\rho,u)\circ\Phi(0,u)=u$ whenever $u\in\SSS_{\rho,\delta}$ and 
 $\Phi(0,u)\circ\Phi(\rho,u)=u$ whenever $u\in\SSS_{0,\delta}$.
 We have thus proven that $\Phi(\rho,\cdot)|_{S_{0,\delta}}$ is an
 omeomorphism from
 $(S_{0,\delta},\Sigma_{0,\delta})$ to 
 $(S_{\rho,\delta},\Sigma_{\rho,\delta})$
 whose inverse is $\Phi(0,\cdot)|_{S_{\rho,\delta}}$. 

Now let:
\begin{align}\label{eqn:def-a-b-rho-1}
 a'_\rho:=&\sup_{(\X_1\oplus\X_2)\cap\Sigma_{\rho,\delta}}f_\rho
 &
 a''_\rho:=&\inf_{(\X_2\oplus\X_3)\cap\SSS_{\rho,\delta}}f_\rho
 \\
 \label{eqn:def-a-b-rho-2}
 b'_\rho:=&\sup_{(\X_1\oplus\X_2)\cap\SSS_{\rho,\delta}}f_\rho
 &
 b''_\rho:=&\inf_{(\X_2\oplus\X_3)\cap\Sigma_{\rho,\delta}}f_\rho.
\end{align}
Notice that, by definition, $a''_\rho\leq b'_\rho$.
For $\rho=0$ it is easy to see that:
\begin{equation*}
 a'_0=
 \lm_{i-1}+
 \frac{\delta^2}{2}(\hat\lm-\lm_{i-1})
 <
 \hat\lm=
 a''_0
 =
 b'_0=
 \hat\lm
 <
 \lm_{k+1}-\frac{\delta^2}{2} (\lm_{k+1}-\hat\lm)
 =
 b''_0
\end{equation*}
(remind that $0<\delta<1$). 
Let $\eps_0>0$ with 
$\eps_0<\hat\lm-\lm_{i-1}$.
We claim that, if $\delta^2(\hat\lm-\lm_{i-1})<2\eps_0$,
then:

\begin{equation}\label{eqn:no-pti-critici-su-frontiera}
 \begin{gathered}
 \mbox{there exists no }u\in\Sigma_{0,\delta} 
 \mbox{ with $u$ lower critical for $\bar f_0$ and }
 \ 
 \lm_{i-1}+\eps_0
 \leq f_0(u)
 .
 \end{gathered}
\end{equation}
By contradiction assume that such a $u$ exists; then there exist 
$\lm\in\real$ and $\mu\geq0$ such that  
\eqref{eqn:u-critico-su-Cdelta} holds.
Let $u_i=\Pi_i(u)$, $i=1,2,3$. Taking $v=u_2$ in 
\eqref{eqn:u-critico-su-Cdelta} (with $\rho=0$) yields:
\begin{equation*}
  \hat\lm\norm{u_2}_\LLL^2=
  \scal{Au_2}{u_2}=
  \scal{Au}{u_2}=
  \lm\scal{u}{u_2}_\LLL+\mu\scal{u_2}{u_2}_\LLL=
  (\lm+\mu)\norm{u_2}_\LLL^2.
\end{equation*}
Since $\norm{u_2}_\LLL=\delta>0$, we have $\lm+\nu=\hat\lm$,
so $\lm=\hat\lm-\nu\leq\hat\lm$.
Taking $v=u_3$:
\begin{equation*}
 \lm_{k+1}\norm{u_3}^2\leq
 \scal{Au_3}{u_3}=
 \scal{Au}{u_3}=
 \scal{\lm u+\mu u_2}{u_3}=
 \lm\norm{u_3}^2_\LLL\leq 
 \hat\lm\norm{u_3}_\LLL^2.
\end{equation*}
Since $\lm_{k+1}<\hat\lm$, we have $u_3=0$. 
Then $u\in\X_1\oplus\X_2\cap\Sigma_{0,\delta}$, 
which implies 
\begin{equation*}
 f_0(u)\leq a'_0=
 \lm_{i-1}+
 \frac{\delta^2}{2}(\hat\lm-\lm_{i-1})<
 \lm_{i-1}+\eps_0
\end{equation*}
which gives  a contradiction.
Hence the claim   is proven.  Notice that \eqref{eqn:no-pti-critici-su-frontiera} implies that the only critical
value $\lm_0$ of $\bar f_0$, with 
$\lm_{i-1}+\eps_0\leq\lm_0\leq\lm_{k+1}-\eps_0$, is 
$\lm_0=\hat\lm$. Indeed assume  $u_0$ to be a critical 
point with $\bar f_0(u_0)=\lm_0$: then, by \eqref{eqn:no-pti-critici-su-frontiera}, $u_0\notin\Sigma_{0,\delta}$ 
so \eqref{eqn:u-critico-su-Cdelta} holds
with $\mu=0$ which easily implies $\lm_0=\hat\lm$.

From now on we fix $\eps_0>0$ such that
$5\eps_0<\min(\hat\lm-\lm_{i-1},\lm_{k+1}-\hat\lm)$ and 
$\delta>0$ such that $\delta^2(\hat\lm-\lm_{i-1})\leq\eps_0$
(so \eqref{eqn:no-pti-critici-su-frontiera} holds with $\eps_0/2$).
Using 
\eqref{eqn:estimates-h-h1} we can derive that,  
given $\eps\in]0,\eps_0]$ there exists 
$\rho(\eps)\in]0,\bar\rho]$ such that, if $\rho\in]0,\rho(\eps)]$:
\begin{gather}
 \begin{gathered}
 a'_\rho\leq
 \lm_{i-1}+\eps_0<
 \hat\lm-4\eps<
 \hat\lm-\eps 
 \leq
 a''_\rho
 \leq
 \inf_{\X_2\cap\SSS_{\rho,\delta}}f_\rho
 \leq\qquad\qquad
 \\
 \qquad\qquad\leq
 \sup_{\X_2\cap\SSS_{\rho,\delta}}f_\rho
 \leq 
 b'_\rho 
 \leq
 \hat\lm+\eps< 
 \hat\lm+4\eps<
 \lm_{k+1}-\eps_0\leq
  b''_\rho;
 \end{gathered}\label{eqn:inequalities-a-b-rho}
\\
  \begin{gathered}\label{eqn:non-critical-points-on-Sigma-delta}
  \mbox{there are no $u\in\Sigma_{\rho,\delta}$ with $u$ lower critical for $\bar f_\rho$ and}
  \\
  f_\rho(u)\in\left[\lm_{i-1}+\eps_0,\lm_{k+1}-\eps_0\right];
  \end{gathered}
\\
  \begin{gathered}\label{eqn:non-critical-points-outside-hat-lambda}
  \mbox{there are no $u\in\SSS_{\rho,\delta}$ with $u$ lower critical for $\bar f_\rho$ and}
  \\
  f_\rho(u)\in\left[\lm_{i-1}+\eps_0,\hat\lm-\eps\right]\cup\left[\hat\lm+\eps,\lm_{k+1}-\eps_0\right];
  \end{gathered}
  \\ \label{eqn:stima-Phi-rho}
  \left|f_0(u)-f_\rho(\Phi(\rho,u))\right|<\eps 
  \qquad\qquad
  \forall u\in\SSS_{0,\delta}\mbox{ with } f_0(u)\leq\lm_{k+1}-\eps_0;
   \\ \label{eqn:stima-Phi-zero}
   \left|f_\rho(u)-f_0(\Phi(0,(u))\right|<\eps 
   \qquad\qquad
   \forall u\in\SSS_{\rho,\delta}\mbox{ with } f_\rho(u)\leq\lm_{k+1}-\eps_0.
\end{gather}
To prove \eqref{eqn:stima-Phi-rho} and \eqref{eqn:stima-Phi-zero} we use 
\eqref{eqn:sublevels-bounded}.
If $\sm\in[\eps,4\eps]$, set 
$A_\rho^\sm:=\bar f_\rho^{\hat\lm/2-\sm}$,
$B_\rho^\sm:=\bar f_\rho^{\hat\lm/2+\sm}$ i.e.:
\begin{equation*}
 A_\rho^\sm=
 \set{u\in S_{\rho,\delta}\st f_\rho(u)\leq\hat\lm/2-\sm},
 B_\rho^\sm=
 \set{u\in S_{\rho,\delta}\st f_\rho(u)\leq\hat\lm/2+\sm}.
\end{equation*}
Moreover set $\tilde A_\rho^\sm:=\Phi(\rho,A_0^\sm)$,
$\tilde B_\rho^\sm:=\Phi(\rho,B_0^\sm)$,
$\hat A_\rho^\sm:=\Phi(0,A_\rho^\sm)$,
$\hat B_\rho^\sm:=\Phi(0,B_\rho^\sm)$.
From \eqref{eqn:stima-Phi-rho} and \eqref{eqn:stima-Phi-zero} (remind that $\Phi(\rho,\cdot)^{-1}=\Phi(0,\cdot)$) we get:
\begin{gather*}
 A_0^{4\eps}\subset\Phi(0,A^{3\eps}_\rho)
 \subset A_0^{2\eps}\subset\Phi(0,A^{\eps}_\rho)
 \quad,\quad
 B_0^{\eps}\subset\Phi(0,B^{2\eps}_\rho)
 \subset B_0^{3\eps}\Phi(0,B^{4\eps}_\rho),
 \\
 A_\rho^{4\eps}\subset\Phi(\rho,A^{3\eps}_0)
 \subset A_\rho^{2\eps}\subset\Phi(\rho,A^{\eps}_0)
 \quad,\quad 
 B_\rho^{\eps}\subset\Phi(\rho,B^{2\eps}_0)
 \subset B_\rho^{3\eps}\Phi(\rho,B^{4\eps}_0).
\end{gather*}

The above inclusions give rise to the following diagram in homology:
\begin{equation*}
\begin{tikzcd}
 H_q(B_{\rho}^{\eps},A_{\rho}^{4\eps})
 \arrow[r,"i_1^*"]\arrow[d,"\phi_1^*"]&
 H_q(\tilde B_{\rho}^{2\eps},\tilde A_{\rho}^{3\eps})
 \arrow[r,"i_2^*"]&
 H_q(B_{\rho}^{3\eps},A_{\rho}^{2\eps})
 \arrow[r,"i_3^*"]\arrow[d,"\phi_3^*"]&
 H_q(\tilde B_{\rho}^{4\eps},\tilde A_{\rho}^{\eps})
\\
 H_q(\hat B_{\rho}^{\eps},\hat A_{\rho}^{4\eps})
 \arrow[r,"j_1^*"]&
 H_q(B_{0}^{2\eps},A_{0}^{3\eps})
 \arrow[r,"j_2^*"]\arrow[u,"\phi_2^*"]&
 H_q(\hat B_{\rho}^{3\eps},\hat A_{\rho}^{2\eps})
 \arrow[r,"j_3^*"]&
 H_q(B_{0}^{4\eps},A_{0}^{\eps})\arrow[u,"\phi_2^*"]
\end{tikzcd}
\end{equation*}
where $i_1$, $i_2$, $i_3$,  $j_1$, $j_2$, $j_3$ are embeddings and $\phi_1$,  $\phi_3$
are restrictions of $\Phi(0,\cdot)$, while $\phi_2$,
$\phi_4$
are restrictions of $\Phi(\rho,\cdot)$.  It is clear that
$\phi^*_i$ are isomorphisms.
Notice that $i_2\circ\phi_2\circ j_1\circ\phi_1$
is the 
embedding of $(B_\rho^\eps,A_\rho^{4\eps})$ in
 $(B_\rho^{3\eps},A_\rho^{2\eps})$  and
 $j_3\circ\phi_3\circ i_2\circ\phi_2$
  is the 
embedding of $(B_0^{2\eps},A_0^{3\eps})$ in
 $(B_0^{4\eps},A_0^{\eps})$ 

Since there are no critical values for $\bar f_0$ in
$[\hat\lm-4\eps,\hat\lm-\eps]\cup[\hat\lm+\eps,\hat\lm+4\eps]$ (see \eqref{eqn:non-critical-points-outside-hat-lambda}), then the pair 
$(B_\rho^{\eps},A_\rho^{4\eps})$ 
is a deformation retract of the pair 
$(B_\rho^{3\eps},A_\rho^{2\eps})$, so 
$i_2^*\circ \phi_2^*\circ j_1^*\circ\phi_1^*$ is an isomorphism. 
For analoguous reasons  
$j_3^*\circ\phi_3^*\circ i_2^*\circ\phi_2^*$ is an 
isomorphism. It follows that 
$i_2^*\circ\phi_2^*:H_q(B_0^{2\eps},A_0^{3\eps})\to H_q(B_\rho^{3\eps},A_\rho^{2\eps})$ is an isomorphism.

From the definitions \eqref{eqn:def-a-b-rho-1} and \eqref{eqn:def-a-b-rho-2} we have the inclusions:
\begin{equation*}
 (\SSS_{0,\delta}\cap(\X_1\oplus\X_2),\Sigma_{0,\delta}\cap(\X_1\oplus\X_2))\subset
 (B_\rho^{3\eps},A_\rho^{2\eps})\subset
 (\SSS_{0,\delta}\setminus\X_3,\SSS_{\rho,\delta}\setminus(\X_2\oplus\X_3))
\end{equation*}
which allow to repeat the arguments of \cite{BenCanBifur}
(see also the proof of Lemma 2.3 in \cite{MarSacSNS1997}).
to estimate the relative category:
\begin{equation*}
 \mathrm{cat}_{(B_\rho^{3\eps},A_\rho^{2\eps})}(B_\rho^{3\eps})\geq2
 \qquad\forall \rho\in]0,\rho(\eps)].
\end{equation*}
This implies that $\bar f_\rho$ hat at least two critical points $\bar u_{1,\rho}$, $\bar u_{2,\rho}$ with 
$\hat\lm-3\eps\leq f_\rho(\bar u_{i,\rho})\leq\hat\lm+2\eps$.  We have
$\norm{\bar u_{i,\rho}}_\LLL^2/2+\hh_{1,\rho}(\bar u_{i,\rho})=1$ and:
\begin{equation}\label{eqn:equation-rho}
  \scal{A\bar u_{i,\rho}+\nabla \hh_\rho(\bar u_{i,\rho})}{v}=\lm_{i,\rho}
  \scal{\bar u_{i,\rho}+\nabla_\LLL \hh_{1,\rho}(\bar u_{i,\rho})}{v}_\LLL
  \qquad\forall v\in \HHH.
\end{equation}
for a suitable Lagrange multiplier $\lm_{i,\rho}\in\real$
(there is no $\mu$, due to \eqref{eqn:non-critical-points-on-Sigma-delta}).
Taking $v=\bar u_{i,\rho}$ in \eqref{eqn:equation-rho}:
\begin{multline*}
 \left[\hat\lm-2\eps,\hat\lm+3\eps\right]\ni
 \bar f(\bar u_{i,\rho})=
 \frac12\scal{A\bar u_{i,\rho}}{\bar u_{i,\rho}}+
 \hh_{\rho}(\bar u_{i,\rho})=
 \\
 \hh_{\rho}(\bar u_{i,\rho})-
 \frac12\scal{\nabla \hh_{\rho}(\bar u_{i,\rho})}{\bar u_{i,\rho}}+
 \frac{\lm_{i,\rho}}{2}
 \left(
 \norm{\bar u_{i,\rho}}_\LLL^2+
 \scal{\nabla_\LLL \hh_{1,\rho}(\bar u_{i,\rho})}{\bar u_{i,\rho}}_\LLL
 \right)=
 \\
 \underbrace{
 \hh_{\rho}(\bar u_{i,\rho})-
 \frac{\scal{\nabla \hh_{\rho}(\bar u_{i,\rho})}{\bar u_{i,\rho}}}{2}
 }_{:=C_1(\rho)}+
 \lm_{i,\rho}
  \left(
  1+
  \underbrace{
 \left(
 \frac{\scal{\nabla_\LLL \hh_{1,\rho}(\bar u_{i,\rho})}{\bar u_{i,\rho}}_\LLL}{2}-
 \hh_{1,\rho}(\bar u_{i,\rho})
 \right)
 }_{:=C_2(\rho)}
  \right)
\end{multline*}
By using \eqref{eqn:estimates-h-h1} we obtain $C_1(\rho)\to0$,  $C_2(\rho)\to0$, so
for $\rho(\eps)$ small enough we have $|\lm_{i,\rho}-\hat\lm|<4\eps$. We have thus proven that $\lm_{1,\rho}\to\hat\lm$ as $\rho\to0$. Let $u_{i,\rho}:=\rho\bar u_{1,\rho}$. Clearly $u_{i,\rho}\xrightarrow{\LLL}0$ as $\rho\to0$.  By multiplying \eqref{eqn:equation-rho} by $\rho$ and using the 
definitions of $\hh_\rho$ and $\hh_{1,\rho}$ we get that $(u,\lm)=(u_{i,\rho},\lm_{i,\rho})$ verify \eqref{eqn:abstract-bifurcation-equation}.
Taking the scalar product with $u_{i,\rho}$ in \eqref{eqn:abstract-bifurcation-equation}
gives $\scal{Au_{i,\rho}}{u_{i,\rho}}\to0$. Then, by 
\eqref{eqn:coercitivity}, we have $u_{i,\rho}\xrightarrow{\HHH}0$.
\end{proof}

\section{A global bifurcation result for radial solutions}

We consider the case $N=2$ and $\Omega=B(0,R)=\set{x\in\real^2\st\|X\|<R}$. We look for radial solutions for Problem \eqref{eqn:z-neumann-problem}, i.e. $z(x,y)=w(\|(x,y)\|)$.
Actually with similar arguments we could have considered the 
general case $N\geq2$.
Given $R>0$, it is therefore convenient to introduce the Hilbert space: 
  \begin{equation*}
  E:=\set{w:[0,R]\to\real\st\int_0^R\rho\dot{w}^2\, d\rho<+\infty}
  \end{equation*}
  endowed with
  $\displaystyle{(v,w)_E:=\int_0^R\rho\dot v\dot w  \,d\rho+\int_0^R\rho v w  \,d\rho}$ and for $\lm>0$ the set:
  \begin{equation*}
   W_\lm:=\set{w\in E\st 1+\sqrt{\lm}w(\rho)>0},\qquad
   \mathcal{W}:=\set{(w,\lm)\in\real\times E\st \lm>0, w\in W_\lm}.
  \end{equation*}
  It is clear that $\|w\|_\infty\leq C\|w\|_E$, for a suitable constant $C$, so $W$ is open in $E$ and $\mathcal{W}$ is open in $\real\times E$. 
As well known the search for radial solutions leads the equation:
\begin{equation}\label{eqn:radial-equation}
\left\{
 \begin{aligned}
 &\ddot{w}+\frac{\dot{w}}{\rho}=-\lm w-\frac{\lm w}{1+\sqrt{\lm}w}=:f_\lm(w),
 \\
 &\dot{w}(0)=\dot{w}(R)=0.
 \end{aligned}
 \right.
\end{equation}
By the above we mean that:
\begin{equation}\label{eqn:radial-equation-weak}
 (w,\lm)\in\mathcal{W},
 \qquad
 \int_0^R\rho\dot{w}\dot{\delta}\,d\rho=
 \int_0^R \rho f_\lm(w)\delta\,d\rho
 \quad
 \forall v\in E.
\end{equation}
It is standard to check that ``weak solutions'', i.e. solutions
to \eqref{eqn:radial-equation-weak} actually solve \eqref{eqn:radial-equation} in a classical sense.

It is clear that $(0,\lm)$ is a solution for \eqref{eqn:radial-equation} for any $\lm\in\real$. We call ``nontrivial '' solution
a pair $(w,\lm)$ with $w\neq0$ such that \eqref{eqn:radial-equation} 
holds.

\begin{rmk}
 If $(w,\lm)$ is a nontrivial solution then $\lm>0$. To see this
 it suffices to multiply \eqref{eqn:radial-equation} by $u$ and
 integrate over $[0,R]$. Actually this property is true in the 
 general case (not just in the radial problem).
\end{rmk}
We shall use the following simple inequality.
\begin{rmk}\label{rmk:inequality-logarithm} 
 Let $0<a<b<+\infty$. We have:
 \begin{equation}\label{eqn:inequality-logarithm} 
  \frac{b-a}{b}\leq\ln\left(\frac{b}{a}\right)\leq\frac{b-a}{a}.
 \end{equation}
 We have indeed:
 \begin{equation*}
  \ln\left(\frac{b}{a}\right)=
  \ln\left(1+\frac{b-a}{a}\right)\leq\frac{b-a}{a}
 \end{equation*}
 and\begin{equation*}
  \ln\left(\frac{b}{a}\right)=
  -\ln\left(\frac{a}{b}\right)=
  -\ln\left(1+\frac{a-b}{b}\right)\geq
  -\frac{a-b}{b}=\frac{b-a}{b}.
 \end{equation*}
\end{rmk}

\begin{figure}[hb]
\centerline{
 \includegraphics[height=3cm]{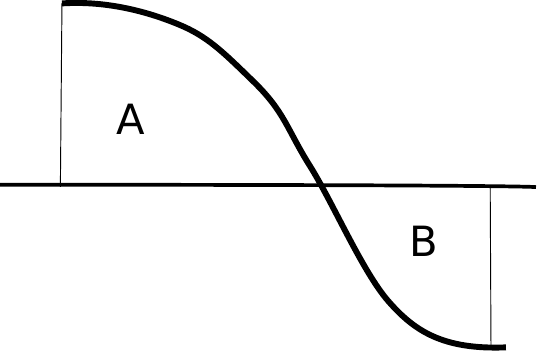}
\qquad\qquad
 \includegraphics[height=3cm]{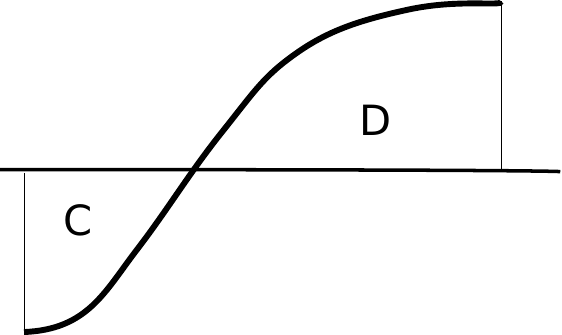}
}
\caption{The different cases}
\label{figureABCD}
\end{figure}

Now let us suppose  that a solution $(w,\lm)$ exists
so we can find some properties and estimates on $w$.
Arguing as in the proof of Lemma 2.2 in \cite{CrandallRabinowitzSturmLiou1970} we have  that either $w=0$ or $[0,R]$ can be split as the union of
a finite number of subintervals $[r_{1,i},r_{2,i}]$,
$i=1\dots,k$, where $w$ has one of the following behaviors (see figure \ref{figureABCD}, we are skipping the index $i$): 
\begin{description}
 \item[(A)] $w(r_1)>0$, $\dot{w}(r_1)=0$,
$\dot{w}<0$ in $]r_1,r_2]$, and $w(r_2)=0$; 
 \item[(B)] $w(r_1)=0$,
$\dot{w}<0$ in $[r_1,r_2[$, $\dot{w}(r_2)=0$, and $w(r_2)<0$; 
 \item[(C)] $w(r_1)<0$, $\dot{w}(r_1)=0$,
$\dot{w}>0$ in $]r_1,r_2]$, and $w(r_2)=0$; 
 \item[(D)] $w(r_1)=0$,
$\dot{w}>0$ in $[r_1,r_2[$, $\dot{w}(r_2)=0$, and $w(r_2)>0$.
\end{description}

So let $w:[r_1,r_2]\to\real$ be as in one of the above cases.
Multiplying \eqref{eqn:radial-equation} by $\dot{w}$ gives
\begin{equation*}
 \frac{1}{2}\ddot{w}\dot{w}'+\frac{\dot{w}^2}{\rho}=\frac{d}{d\rho}F_\lm(w)
\end{equation*}
where 
\begin{equation*}
  F_\lm(s)=\ln(1+\sqrt{\lm}s)-\sqrt{\lm}s-\frac\lm2s^2.
\end{equation*}
Let $p:=\dot{w}^2$ the previous equation can be written as:
\begin{equation*}
 \frac{1}{2}\dot{p}+\frac{p}{\rho}=\frac{d}{d\rho}F_\lm(w)
\end{equation*}
which is equivalent to
\begin{equation*}
 \frac{d}{d\rho}(\rho^2p)=2\rho^2\frac{d}{d\rho}F_\lm(w)\rho^2=
 2\rho^2\frac{d}{d\rho}F_1\left(\sqrt{\lm}w\right).
\end{equation*}
We integrate  between $\rho_1$ and $\rho_2$, where 
$r_1\leq\rho_1\leq\rho_2\leq r_2$:
\begin{equation*}
  \rho_2^2p(\rho_2)- \rho_1^2p(\rho_1)=
  2\rho_2^2F_\lm(w(\rho_2))-2\rho_1^2F_\lm(w(\rho_1))-
  \int_{\rho_1}^{\rho_2}4\sm F_\lm(w(\sm))\,d\sm
\end{equation*}
Notice that $F_\lm$ is increasing on $\left]-\dfrac{1}{\sqrt{\lm}},0\right[$ and decreasing 
on $]0,+\infty[$, so:
\begin{equation*}
 \sm\mapsto F_\lm(w(\sm))\mbox{ is increasing (decreasing) in cases \thetag{A} and \thetag{C} (in cases \thetag{B} and \thetag{D})}
\end{equation*}
We hence get, in cases \thetag{A} and \thetag{C}:
\begin{equation*}
 -2(\rho_2^2-\rho_1^2)F_\lm(w(\rho_2))\leq
 -\int_{\rho_1}^{\rho_2}4\sm F_\lm(w(\sm))\,d\sm\leq 
 -2(\rho_2^2-\rho_1^2)F_\lm(w(\rho_1))
\end{equation*}
while in cases \thetag{B} and \thetag{D}:
\begin{equation*}
 -2(\rho_2^2-\rho_1^2)F_\lm(w(\rho_1))\leq
 -\int_{\rho_1}^{\rho_2}4\sm F_\lm(w(\sm))\,d\sm\leq 
 -2(\rho_2^2-\rho_1^2)F_\lm(w(\rho_2))
\end{equation*}
So in cases \thetag{A} and \thetag{C} we have:
\begin{equation}\label{eqn:inequality-AC}
 2\rho_1^2(F_\lm(w(\rho_2))-F_\lm(w(\rho_1))\leq
 \rho_2^2p(\rho_2)- \rho_1^2p(\rho_1)\leq
 2\rho_2^2(F_\lm(w(\rho_2))-F_\lm(w(\rho_1))
\end{equation}
and in cases \thetag{B} and \thetag{D}:
\begin{equation}\label{eqn:inequality-BD}
 2\rho_2^2(F_\lm(w(\rho_2))-F_\lm(w(\rho_1))\leq
 \rho_2^2p(\rho_2)- \rho_1^2p(\rho_1)\leq
 2\rho_1^2(F_\lm(w(\rho_2))-F_\lm(w(\rho_1))
\end{equation}
Now we estimate $w(\rho)$ - we need to take into account all the four cases \thetag{A},\thetag{B},\thetag{C},\thetag{D}.

Case \thetag{A}. 
We  rename $\bar\rho:=r_1$, $\rho_0:=r_2$ and let $h:=w(\bar\rho)>0$. We use \eqref{eqn:inequality-AC} with 
$\rho_1=\bar\rho$ and $\rho_2=\sm\in[\bar\rho,\rho_0]$:
\begin{equation*}
 2\bar\rho^2(F_\lm(w(\sm))-F_\lm(h))\leq
 \sm^2\dot{w}(\sm)^2\leq
 2\sm^2(F_\lm(w(\sm))-F_\lm(h)).
\end{equation*}
Then we take the square root and divide: 
\begin{equation*}
 \sqrt{2}\frac{\bar\rho}{\sm}\leq
 \frac{-\dot{w}(\sm)}{\sqrt{F_\lm(w(\sm))-F_\lm(h)}}\leq
 \sqrt{2}
\end{equation*}
and now we integrate
between $\bar\rho$ and $\rho\in[\bar\rho,\rho_0]$ getting:
\begin{equation*}
 \sqrt{2}\bar\rho\ln\left(\frac{\rho}{\bar\rho}\right)\leq
 -\Phi_{\lm,h}(w(\rho))+\Phi_{\lm,h}(h)\leq
 \sqrt{2}(\rho-\bar\rho)
\end{equation*}
where $\Phi_{\lm,h}:[0,h]\to\real$ is defined by:
\begin{equation*}
 \Phi_{\lm,h}(s):=\int_0^s\frac{d\xi}{\sqrt{F_\lm(\xi)-F_\lm(h)}}
\end{equation*}
(it is simple to check the the integral converges at $\xi=h$).  So we deduce:
\begin{equation*}
 \Phi_{\lm,h}^{-1}\left(\Phi_{\lm,h}(h)-\sqrt{2}\left(\rho-\bar\rho\right)\right)
 \leq w(\rho)\leq
 \Phi_{\lm,h}^{-1}\left(\Phi_{\lm,h}(h)-\sqrt{2}\bar\rho\ln\left(\frac{\rho}{\bar\rho}\right)\right)
\end{equation*}
which we prefer to write as
\begin{equation}\label{eq-inequality-case-AB}
 \Phi_{\lm,h}^{-1}\left(\Phi_{\lm,h}(h)+\sqrt{2}\left(\bar\rho-\rho\right)\right)
 \leq w(\rho)\leq
 \Phi_{\lm,h}^{-1}\left(\Phi_{\lm,h}(h)+\sqrt{2}\bar\rho\ln\left(\frac{\bar\rho}{\rho}\right)\right)
\end{equation}
In particular, taking $\rho=\rho_0$, which gives $w(\rho_0)=0$, 
(and using \eqref{eqn:inequality-logarithm}) we have:
\begin{equation}\label{eqn:estimate-interval-case-A}
\sqrt{2}\frac{\bar\rho}{\rho_0}(\rho_0-\bar\rho)\leq
 \sqrt{2}\bar\rho\ln\left(\frac{\rho_0}{\bar\rho}\right)\leq\Phi_{\lm,h}(h)\leq\sqrt{2}\left(\rho_0-\bar\rho\right).
\end{equation}
Moreover taking $\rho_1=\bar\rho$ and $\rho_2=\rho_0$ in \eqref{eqn:inequality-AC} we have:
\begin{equation}\label{eqn:estimate-derivative-case-A}
 \sqrt{2}\frac{\bar\rho}{\rho_0}\sqrt{-F_\lm(h))}\leq
 -\dot{w}(\rho_0)\leq
 \sqrt{2}\sqrt{-F_\lm(h))}
\end{equation}
 
Case \thetag{B}. 
We rename $\rho_0:=r_1$, $\bar\rho:=r_2$ and let $h:=w(\bar\rho)<0$.
We use \eqref{eqn:inequality-BD} with 
$\rho_1=\sm\in[\rho_0,\bar\rho]$ and $\rho_2=\bar\rho$:
\begin{equation*}
 2\bar\rho^2(F_\lm(h)-F_\lm(w(\sm)))\leq
 -\sm^2\dot{w}(\sm)^2\leq
 2\sm^2(F_\lm(h)-F_\lm(w(\sm))).
\end{equation*}
We change sign and proceed as in   case \thetag{A}:
\begin{equation*}
 2\sm^2(F_\lm(w(\sm))-F_\lm(h))\leq
 \sm^2\dot{w}(\sm)^2\leq
 2\bar\rho^2(F_\lm(w(\sm))-F_\lm(h))).
\end{equation*}
Take the square root and divide: 
\begin{equation*}
 \sqrt{2}\leq
 \frac{-\dot{w}(\sm)}{\sqrt{F_\lm(w(\sm))-F_\lm(h)}}\leq
 \sqrt{2}\frac{\bar\rho}{\sm}.
\end{equation*}
Integrate on $[\rho,\bar\rho_0]$:
\begin{equation*}
 \sqrt{2}(\bar\rho-\rho)\leq
 -\Phi_{\lm,h}(h)+\Phi_{\lm,h}(w(\rho))\leq
 \sqrt{2}\bar\rho\ln\left(\frac{\bar\rho}{\rho}\right)
\end{equation*}
defining
$\Phi_{\lm,h}:[h,0]\to\real$ as in case \thetag{A}.
Applying $\Phi_{\lm,h}^{-1}$ we get that 
\eqref{eq-inequality-case-AB} holds in case \thetag{B} too.
In particular, taking $\rho=\rho_0$ (and using \eqref{eqn:inequality-logarithm}):
\begin{equation}\label{eqn:estimate-interval-case-B}
 \sqrt{2}(\bar\rho-\rho_0)\leq
 -\Phi_{\lm,h}(h)\leq
 \sqrt{2}\bar\rho\ln\left(\frac{\bar\rho}{\rho_0}\right)\leq
 \sqrt{2}\frac{\bar\rho}{\rho_0}(\bar\rho-\rho_0)
\end{equation}
and taking $\rho_1=\rho_0$ and $\rho_2=\bar\rho$ in \eqref{eqn:inequality-BD} we have:
\begin{equation}\label{eqn:estimate-derivative-case-B}
 \sqrt{2}\sqrt{-F_\lm(h))}\leq
 -\dot{w}(\rho_0)\leq
 \sqrt{2}\frac{\bar\rho}{\rho_0}\sqrt{-F_\lm(h))}
\end{equation}

Case \thetag{C}. 
We rename $\bar\rho:=r_1$, $\rho_0:=r_2$ end let $h:=w(\bar\rho)<0$.
Using \eqref{eqn:inequality-AC} with 
$\rho_1=\bar\rho$ and $\rho_2=\sm\in[\bar\rho,\rho_0]$ we obtain the same
inequality of case \thetag{A}. 
After taking the square root and dividing: 
\begin{equation*}
 \sqrt{2}\frac{\bar\rho}{\sm}\leq
 \frac{\dot{w}(\sm)}{\sqrt{F_\lm(w(\sm))-F_\lm(h)}}\leq
 \sqrt{2}.
\end{equation*}
We integrate
between $\bar\rho$ and $\rho\in[\bar\rho,\rho_0]$ getting:
\begin{equation*}
 \sqrt{2}\bar\rho\ln\left(\frac{\rho}{\bar\rho}\right)\leq
 \Phi_{\lm,h}(w(\rho))-\Phi_{\lm,h}(h)\leq
 \sqrt{2}(\rho-\bar\rho)
\end{equation*}
with $\Phi_{\lm,h}:[h,0]\to\real$ defined as above. So we deduce:
\begin{equation}\label{eq-inequality-case-CD}
 \Phi_{\lm,h}^{-1}\left(\Phi_{\lm,h}(h)+\sqrt{2}\bar\rho\ln\left(\frac{\rho}{\bar\rho}\right)\right)
 \leq w(\rho)\leq
 \Phi_{\lm,h}^{-1}\left(\Phi_{\lm,h}(h)+\sqrt{2}\left(\rho-\bar\rho\right)\right)
\end{equation}
In particular, taking $\rho=\rho_0$ (and using \eqref{eqn:inequality-logarithm}):
\begin{equation}\label{eqn:estimate-interval-case-C}
 \sqrt{2}\frac{\bar\rho}{\bar\rho_0}(\rho_0-\bar\rho)\leq
 \sqrt{2}\bar\rho\ln\left(\frac{\rho_0}{\bar\rho}\right)\leq-\Phi_{\lm,h}(h)\leq\sqrt{2}\left(\rho_0-\bar\rho\right).
\end{equation}
Moreover taking $\rho_1=\bar\rho$ and $\rho_2=\rho_0$ in \eqref{eqn:inequality-AC}  we have:
\begin{equation}\label{eqn:estimate-derivative-case-C}
 \sqrt{2}\frac{\bar\rho}{\rho_0}\sqrt{-F_\lm(h))}\leq
 \dot{w}(\rho_0)\leq
 \sqrt{2}\sqrt{-F_\lm(h))}
\end{equation}
 
Case \thetag{D}.
We rename $\rho_0:=r_1$, $\bar\rho:=r_2$ and let $h:=w(\bar\rho)>0$. 
Using \eqref{eqn:inequality-BD} with 
$\rho_1=\sm\in[\rho_0,\bar\rho]$ and $\rho_2=\bar\rho$
we obtain the same inequalities of case \thetag{B}. When we take the square root and divide:
\begin{equation*}
 \sqrt{2}\leq
 \frac{\dot{w}(\sm)}{\sqrt{F_\lm(w(\sm))-F_\lm(h)}}\leq
 \sqrt{2}\frac{\bar\rho}{\sm}.
\end{equation*}
Integrate on $[\rho,\bar\rho_0]$:
\begin{equation*}
 \sqrt{2}(\bar\rho-\rho)\leq
 \Phi_{\lm,h}(h)-\Phi_{\lm,h}(w(\rho))\leq
 \sqrt{2}\bar\rho\ln\left(\frac{\bar\rho}{\rho}\right)
\end{equation*}
with the usual definition of
$\Phi_{\lm,h}:[h,0]\to\real$. Applying $\Phi_{\lm,h}^{-1}$ we obtain that
\eqref{eq-inequality-case-CD} holds in case \thetag{D} too.
In particular, taking $\rho=\rho_0$ (and using \eqref{eqn:inequality-logarithm}):
\begin{equation}\label{eqn:estimate-interval-case-D}
 \sqrt{2}\left(\bar\rho-\rho_0\right)\leq\Phi_{\lm,h}(h)\leq\sqrt{2}\bar\rho\ln\left(\frac{\bar\rho}{\bar\rho_0}\right)\leq
 \sqrt{2}\frac{\bar\rho}{\rho_0}(\bar\rho-\rho_0)
\end{equation}
and taking $\rho_1=\rho_0$ and $\rho_2=\bar\rho_0$ in \eqref{eqn:inequality-BD} we have:
\begin{equation}\label{eqn:estimate-derivative-case-D}
 \sqrt{2}\sqrt{-F_\lm(h))}\leq
 \dot{w}(\rho_0)\leq
 \sqrt{2}\frac{\bar\rho}{\rho_0}\sqrt{-F_\lm(h))}
\end{equation}

Now we have:
\begin{multline*}
 \sqrt{2}\Phi_{\lm,h}(h)=
 \\
 \int_0^h\frac{d\xi}{\sqrt{F(\sqrt{\lm}\xi)-F(\sqrt{\lm}h)}}=
 \int_0^1\frac{h\,d\sm}{\sqrt{F(\sm\sqrt{\lm}h)-F(\sqrt{\lm}h)}}=
 \frac{1}{\sqrt{\lm}}\bar\Phi(\sqrt{\lm}h)
\end{multline*}
where:
\begin{equation*}
 \bar\Phi(s):=\int_0^1\frac{s\,d\sm}{\sqrt{F(\sm s)-F(s)}}=
 \mathrm{sgn}(s)\int_0^1\sqrt{\frac{s^2}{F(\sm s)-F(s)}}\,d\sm.
\end{equation*}

With simple computations:
\begin{equation*}
 \lim_{s\to0}\frac{s^2}{F(\sm s)-F(s)}=\frac{1}{1-\sm^2},
 \qquad
 \lim_{s\to+\infty}\frac{s^2}{F(\sm s)-F(s)}=\frac{2}{1-\sm^2},
\end{equation*}
and
\begin{equation*}
 \lim_{s\to-1^-}\frac{s^2}{F(\sm s)-F(s)}=0
\end{equation*}
So we deduce that (see figure \ref{figure:graph}):
\begin{gather}\label{eqn:limits-Phi-positive}
 \lim_{h\to0^+}\Phi_{\lm,h}(h)=
 \frac{\pi}{2\sqrt{2\lm}}
 ,\quad
 \lim_{h\to+\infty}\Phi_{\lm,h}(h)=
 \frac{\pi}{2\sqrt{\lm}},
 \\
 \label{eqn:limits-Phi-negative}
 \lim_{h\to0^-}\Phi_{\lm,h}(h)=
 -\frac{\pi}{2\sqrt{2\lm}}
 ,\quad
 \lim_{h\to-1^+}\Phi_{\lm,h}(h)=0.
\end{gather}

\begin{figure}[hb]
 \centerline{
 \includegraphics[height=5cm]{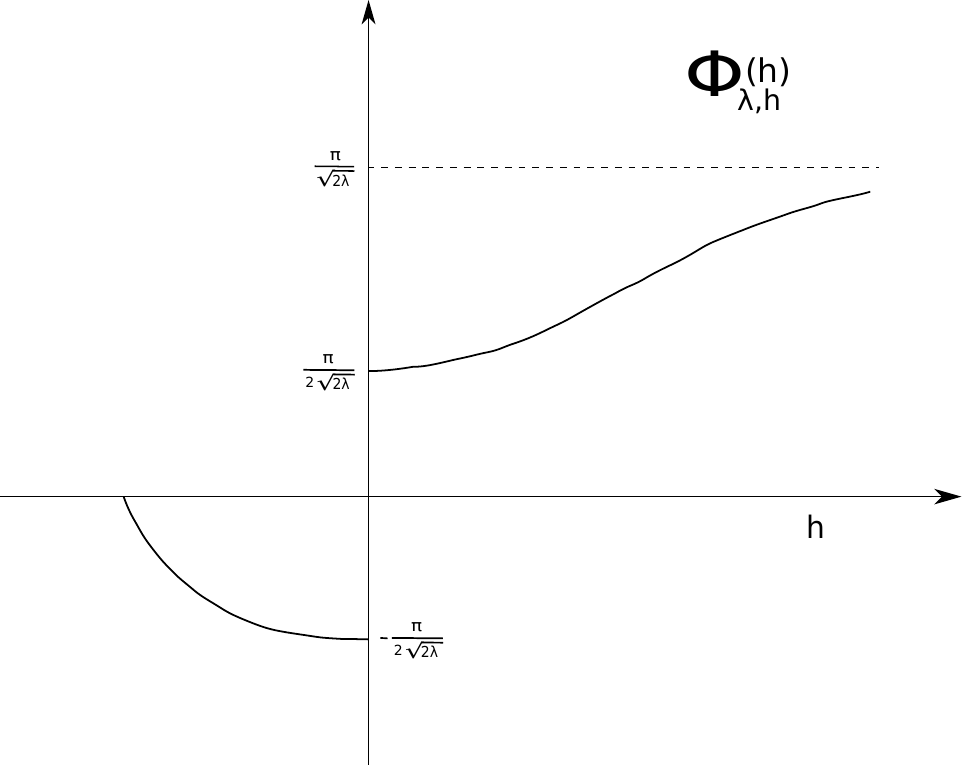}
 }
 \caption{Graph of $\Phi_{\lm,h}(h)$}
 \label{figure:graph}
\end{figure}

  To state the main result we need some notation, which we take
  from \cite{CrandallRabinowitzSturmLiou1970, RabinowitzSturmLiouville70}. For $k\in\nat$, $k\geq1$, we consider
  \begin{gather*}
   \mathcal{S}:=\set{(w,\lm)\in\mathcal{W}\st(w,\lm)\mbox{ is a solution to \eqref{eqn:radial-equation}}}
   \\
   \mathcal{S}^+_k:=\set{(w,\lm)\in\mathcal{S}\st\mbox{ $w$ has $k$ nodes in $]0,R[$}, w(0)>0},
   \\
   \mathcal{S}^-_k:=\set{(w,\lm)\in\mathcal{S}\st\mbox{ $w$ has $k$ nodes in $]0,R[$}, w(0)<0},
  \end{gather*}
  We also consider the two eigenvalue problems:
\begin{equation}\label{eqn:radial-eigenvalue-problem}
 \ddot{w}+\frac{\dot{w}}{\rho}=-\mu w,
 \qquad
 \dot{w}(0)=\dot{w}(R)=0.
\end{equation}
\begin{equation}\label{eqn:radial-eigenvalue-problem-zero}
 \ddot{v}+\frac{\dot{v}}{\rho}=-\nu v,
 \qquad
 \dot{v}(0)=0,{v}(R)=0.
\end{equation}
It is clear that $w\neq0$ and $\mu\neq0$  solve \eqref{eqn:radial-eigenvalue-problem} if and only if, for some integer $k\geq1$:
\begin{equation}\label{eqn:radial-eigenvalues}
 \mu=\mu_k:=\left(\frac{y_k}{R}\right)^2
\end{equation}
where $y_k$ denotes the $k$-th nontrivial zero of $J_0'$ and $J_0$ is the first Bessel function, and

\begin{equation}\label{eqn:radial-eigenfunctions}
 w=\alpha w_k,\quad\alpha\in\real,\qquad
 w_k(\rho):=J_0\left(\frac{y_k}{R}\rho\right).
\end{equation}
For the sake of completeness we can agree that $\mu_0=0$ and $w_0(\rho)=J_0(0)$. In the same way
$v\neq0$ and $\nu$  solve \eqref{eqn:radial-eigenvalue-problem-zero} 
if and only if, for some integer $k\geq1$:
\begin{equation}\label{eqn:radial-eigenvalues-zero}
 \nu=\nu_k:=\left(\frac{z_k}{R}\right)^2
\end{equation}
where $z_k$ is the $k$-th zero of $J_0$ and

\begin{equation}\label{eqn:radial-eigenfunctions-zero}
 v=\alpha v_k,\quad\alpha\in\real,\qquad
 v_k(\rho):=J_0\left(\frac{z_k}{R}\rho\right).
\end{equation}
Notice that $\nu_k<\mu_k<\nu_{k+1}$ for all $k$.

\begin{thm}\label{thm:global-bifurvation-radial}
 Let $\mu_k>0$ be an eigenvalue for \eqref{eqn:radial-eigenvalue-problem}. Then $\mathcal{S}_k^+$ is a  connected set  and
 \begin{itemize}
  \item $(0,\mu_k/2)\in\overline{\mathcal{S}_k^+}$;
  \item 
  $\displaystyle{
    0<
    \inf\set{\lm\in\real\st\exists w\in E\mbox{ with }(w,\lm)\in\mathcal{S}_k^+}
    }$;
  \item
    $\displaystyle{
    \sup\set{\lm\in\real\st\exists w\in E\mbox{ with }(w,\lm)\in\mathcal{S}_k^+}
    <+\infty
    }$;
  \item
   $\mathcal{S}_k^+$ is unbounded and contains a sequence $(w_n,\lm_n)$
   such that $\|w_n\|_E\to\infty$ and 
   \begin{equation}\label{eqn:limit-of-the-branch}
    \lim_{n\to\infty}\lm_n=
    \begin{cases}
     \mu_{k/2}&\mbox{ if $k$ is even},
     \\
     \nu_{(k+1)/2}&\mbox{ if $k$ is odd}.
    \end{cases}
   \end{equation}
   
 \end{itemize}
\end{thm}
Figure \ref{figure:bifur-diag} somehow illustrates Theorem \eqref{thm:global-bifurvation-radial}.
  \begin{figure}[hb]
  \centerline{
    \includegraphics[height=5cm]{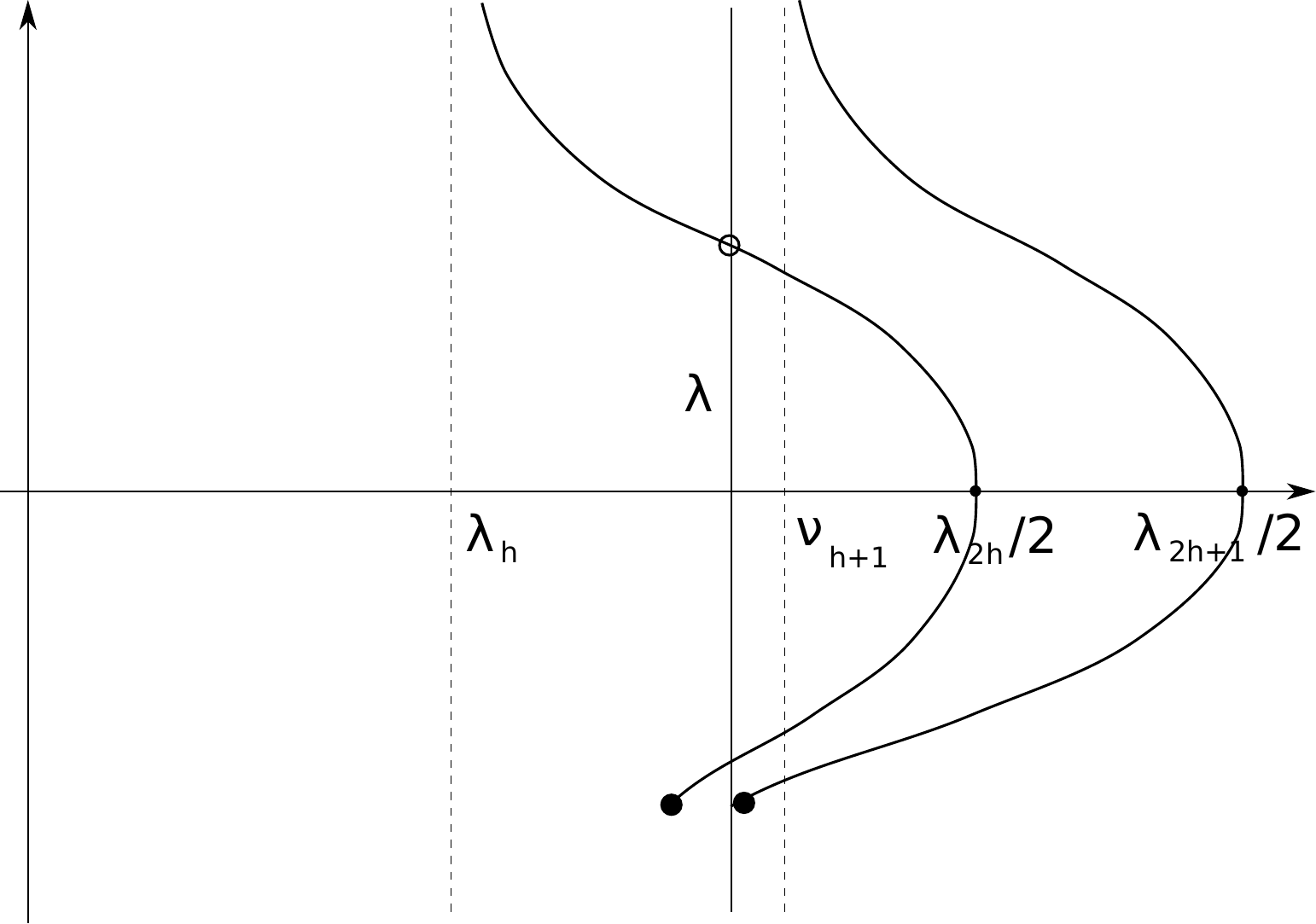}
    }
    \caption{Bifurcation diagram}
    \label{figure:bifur-diag}
  \end{figure}

  The proof of \eqref{thm:global-bifurvation-radial}
  will be obtained from some preliminary statements.
  \begin{rmk}
   If $(w,\lm)\in\mathcal{S}^+$ ( resp. $(w,\lm)\in\mathcal{S}^+$),
   and $0=\rho_0<\rho_1,\dots,\rho_k<\rho_{k+1=R}$, $\rho_1,\dots,\rho_k$ being the nodal points of $w$, then:
   \begin{equation}\label{eqn:estimate-positive-intervals}
    \rho_{i+1}-\rho_i\geq\ (\,\leq\,)\ 
    \frac{\pi}{4\sqrt{\lm}}
    \ \mbox{for $i$ even}
    \quad(\mbox{resp. for $i$ odd}).
   \end{equation}
  This is easily seen using the right hand sides of the inequalities \eqref{eqn:estimate-interval-case-A},
  \eqref{eqn:estimate-interval-case-C},
  and \eqref{eqn:limits-Phi-positive}.
  \end{rmk}

\begin{lma}
  For any integer $k$ there esist two constants $\underline{\lm}_k$
  and $\overline{\lm}_k$ such that
  \begin{equation}\label{eqn:estimate-lambda-on-solutions}
   (w,\lm)\in\mathcal{S}_k^+\cup\mathcal{S}_k^-
   \Rightarrow
   0<\underline{\lm}_k\leq\lm\leq\overline{\lm}_k<+\infty
  \end{equation}
  \end{lma}
  \begin{proof}
  Take any subinterval $[r_1,r_2]$ as in cases
  \thetag{A}--\thetag{D} and consider the first eigenvalue
  $\bar\mu=\bar\mu(r_1,r_2)$ for the mixed type boundary condition:
  \begin{equation*}
   \left\{
   \begin{aligned}
   -&(\rho\dot{w})'=\mu w\quad\mbox{on }]r_1,r_2[
   \\
   &\dot{w}(r_1)=0,w(r_2)=0\qquad
   (\mbox{resp. }w(r_1)=0,\dot{w}(r_2)=0)
   \end{aligned}
   \right.
  \end{equation*}
  in cases \thetag{A},\thetag{C} (resp. cases \thetag{C},\thetag{D}). We can choose an eigenfunction $\bar e$ corresponding
  to $\bar\mu$ so that $z\bar e>0$ in 
  $]r_1,r_2[$. Multiplying  \eqref{eqn:radial-equation}
  by $\bar e$ and integrating over $[r_1,r_2]$ yields:
  \begin{equation*}
   \bar\mu\int_{r_1}^{r_2}\rho z\bar e\,d\rho=
   \lm\int_{r_1}^{r_2}\rho z\bar e\left(1+\frac{1}{1+\sqrt{\lm}z}\right)\,d\rho.
  \end{equation*}
  This implies:
  \begin{equation*}
   \lm\int_{r_1}^{r_2}\rho z\bar e\,d\rho\leq
   \bar\mu\int_{r_1}^{r_2}\rho z\bar e\,d\rho\leq
   2\lm\int_{r_1}^{r_2}\rho z\bar e\,d\rho
  \end{equation*}
  which gives
  \(
   \dfrac{\bar\mu}{2}\leq\lm\leq\bar\mu.
  \)
  Now since $]r_1,r_2[\subset]0,R[$ we have 
  $\bar\mu\geq\bar\mu[0,R]$. On the other side since $w$ has $k$
  nodal points we can choose $r_1$, $r_2$ such that 
  $r_2-r_1\geq R/k$, which implies 
  $\bar\mu\leq\sup\limits_{b-a=R/k}\bar\mu(a,b)<+\infty$.
  This proves \eqref{eqn:estimate-lambda-on-solutions}
  \end{proof}

  \begin{lma}\label{lma:main-lemma-sequences}
   Let $(w_n,\lm_n)$ be a sequence in $\mathcal{S}_k^+$.
   Then we can consider $0<\rho_{1,n}<\cdots<\rho_{k,n}<R$ to be the nodes of $w_n$ and set $\rho_{0,n}:=0$, $\rho_{k+1,n}:=R$;
   in ths way $wn(\rho)>0$ on $]\rho_1,\rho_{i+1}[$ if $i$ is even
   and $wn(\rho)<0$ on $]\rho_1,\rho_{i+1}[$ if $i$ is odd. 
   The following facts are equivalent:
   \begin{description}
    \item[(a)]
    \begin{equation*}
     \lim_{n\to\infty}\sup_{\rho\in[0,R]}w_n(\rho)=+\infty;
    \end{equation*}
    \item[(b)]
    \begin{equation*}
     \lim_{n\to\infty}\inf_{\rho\in[0,R]}(1+\lm_n w_n(\rho))=0;
    \end{equation*}
    \item[(c)]
    \begin{equation*}
     \lim_{n\to\infty}\sup_{\rho\in[\rho_{i,n},\rho_{i+1,n}]}w_n(\rho)=+\infty
     \quad\mbox{if $i$ is even};
    \end{equation*}
    \item[(d)]
    \begin{equation*}
     \lim_{n\to\infty}\inf_{\rho\in[\rho_{i,n},\rho_{i+1,n}]}(1+\lm_n w_n(\rho))=0
     \quad\mbox{if $i$ is odd};
    \end{equation*}
    \item[(e)]
    \begin{equation*}
     \lim_{n\to\infty}\rho_{1+1,n}-\rho_{i,n}=0
     \quad\mbox{if $i$ is odd};
    \end{equation*}
   \end{description}
   Moreover, if any of the above holds, then \eqref{eqn:limit-of-the-branch}  holds.
  \end{lma}
  \begin{proof}
   We can assume, passing to a subsequence that 
   $\lm_n\to\hat\lm\in[\underline{\lm}_k,\overline{\overline{\lm}_k}]$.
   First notice that for all $i$ even (corresponding to $w>0$)
   we have:
   \begin{equation*}
    \rho_{i+1,n}-\rho_{i,n}\geq\frac{\pi}{4\sqrt{\underline{\lm}_k}}
   \end{equation*}
   as we can infer from \eqref{eqn:estimate-interval-case-A} or \eqref{eqn:estimate-interval-case-D} and the behaviour of $\Phi_{\lm,h}(h)$
   in \eqref{eqn:limits-Phi-positive}.
   
   Let 
   \begin{equation*}
     h_{i,n}:=\max_{\rho_{i,n}\leq\rho_{i+1,n}}w(\rho)
     \mbox{ for $i$ even}, 
     \qquad
     h_{i,n}:=\min_{\rho_{i,n}\leq\rho_{i+1,n}}w(\rho) 
     \mbox{ for $i$ odd}
   \end{equation*}
   Then for any $i$ even:
   \begin{equation*}
    h_{i,n}\to+\infty
    \Leftrightarrow
    \Phi_{\lm_n,h_{i,n}}(h_{i,n})\to\frac{\pi}{2\sqrt{\hat\lm}}
    \Leftrightarrow
    \dot{w}(\rho_{i,n})\to+\infty
    \Leftrightarrow
    \dot{w}(\rho_{i+1,n})\to-\infty.
   \end{equation*}
   This can be deduced from \eqref{eqn:limits-Phi-positive},
   \eqref{eqn:estimate-derivative-case-A}, and \eqref{eqn:estimate-derivative-case-D}. In the same way, using
   \eqref{eqn:limits-Phi-negative}, \eqref{eqn:estimate-derivative-case-B}, and \eqref{eqn:estimate-derivative-case-C} we get that,
   for $i$ odd:
   \begin{equation*}
    1+\sqrt{\lm_n} h_{i,n}\to0
    \Leftrightarrow
    \Phi_{\lm_n,h_{i,n}}(h_{i,n})\to0
    \Leftrightarrow
    \dot{w}(\rho_{i,n})\to-\infty
    \Leftrightarrow
    \dot{w}(\rho_{i+1,n})\to+\infty.
   \end{equation*}
   Now we prove our claims. Let $\bar i\in\set{0,\dots,k}$ with
   $\bar i$ even (resp. odd) and suppose that $h_{\bar i,n}\to+\infty$ 
   (resp. $1+\sqrt{\lm_n} h_{\bar i,n}\to0$). Then $F_{\lm_n}(h_{\bar i,n})\to+\infty$
   (resp. $F_{\lm_n}(h_{\bar i,n})\to-\infty$)
   and by 
   \eqref{eqn:estimate-derivative-case-A},
   \eqref{eqn:estimate-derivative-case-D}
   (
   \eqref{eqn:estimate-derivative-case-B},
   \eqref{eqn:estimate-derivative-case-C}
   ) we get that:
   \begin{equation*}
    \dot{w}_n(\rho_{\bar i,n})\to+\infty,\ 
    \dot{w}_n(\rho_{\bar i+1,n})\to-\infty
    \quad
    (\dot{w}_n(\rho_{\bar i,n})\to-\infty,\ 
    \dot{w}_n(\rho_{\bar i+1,n})\to+\infty)
   \end{equation*}
   which in turn implies:
   \begin{equation*}
   F_{\lm_n}(h_{\bar i-1,n})\to-\infty\mbox{ ( resp. }+\infty),
   \ 
    F_{\lm_n}(h_{\bar i+1,n})\to-\infty\mbox{ ( resp. }+\infty)
   \end{equation*}   
   (with the obvious exceptions when $\bar i-1<0$ or $\bar i+1>k$).
   So we get:
   \begin{equation*}
    1+\sqrt{\lm_n}h_{\bar i-1,n}\to0\ 
    (h_{\bar i-1,n}\to+\infty),
    \ 
    1+\sqrt{\lm_n}h_{\bar i+1,n}\to0\ 
    (h_{\bar i+1,n}\to+\infty).
   \end{equation*}
   This shows that the property $|F_{\lm}(h_{i,n})|\to+\infty$
   ``propagates'' from the $i$-th interval to the previous and to the next one. From this it is easy to deduce that 
   \thetag{a}--\thetag{d} are all equivalent. To prove that they
   are equivalent to \thetag{e} just use \eqref{eqn:estimate-interval-case-A},\eqref{eqn:estimate-interval-case-B},\eqref{eqn:estimate-interval-case-C},\eqref{eqn:estimate-interval-case-D}, depending on the case, noticing that
   $\rho_{1,n}\geq\dfrac{\pi}{4\pi\underline{\lm}_k}$, as from
   \eqref{eqn:estimate-positive-intervals}
   (this would not be possible if we were considering $\mathcal{S}_k^-$).
   
   Finally suppose that $(w_n,\lm_n)$ verifies any of 
    \thetag{a}--\thetag{e}. Then $\|w_n\|_\infty \to+\infty$. 
    Let $\hat w_n:=\dfrac{w_n}{\|w_n\|_\infty}$. We can suppose that
    $\hat w_n\rightharpoonup\hat w$ in $E$ and that:
    \begin{gather*}
     \rho_{1,n}\to\rho_1,
     \ 
     \rho_{2j-1,n}\to\rho_j,\ \rho_{2j,n}\to\rho_j
     \mbox{  $1\leq j\leq k/2$},
     \ 
     \rho_{k,n}\to R
     \mbox{ if $k$ is odd},
    \end{gather*}
   where $0=\rho_0<\rho_1<\cdots<\rho_h<\rho_h+1=R$ and
   $h=\lfloor k/2\rfloor$ (so $\rho_1=R$ when $k=1$).
   It is not difficult to prove that 
   $\hat w(\rho)>0$ in $]\rho_i,\rho_{i+1}[$ if $i=0,\dots,h$,
   $\hat w(\rho_1)=\cdots=\hat w(\rho_h)=0$,
   ${\hat w}'(0)=0$ and ${\hat w}'(R)=0$ is $k$ is even
   while $\hat w(R)=0$ is $k$ is odd. 
   Moreover for any $i=0,\dots,h$:
   \begin{equation*}
      -(\rho\hat{w}')'=\hat\lm\hat w\quad\mbox{on }]\rho_i,\rho_{i+1}[
   \end{equation*}
   Now we can rearrange  $\hat w$ defining
   $\tilde w:=\sum_{j=0}^h(-1)^j\alpha_j\hat w\uno_{[\rho_j,\rho_{j+1}]}$, where $\alpha_1=1$ and 
   $\alpha_{j}\hat{w}'_-(\rho_j)=\alpha_{j+1}\hat{w}'_+(\rho_j)$,
   $j=1,\dots,h$. In this way $(\hat\lm,\tilde w)$ is an
   eigenvalue -- eigenfunction pair relative for 
   problem \eqref{eqn:radial-eigenfunctions} if $k$ is even
   and of \eqref{eqn:radial-eigenfunctions-zero} if
   $k$ is odd. Since $\tilde w$ has $h=k/2$ nodal points for
   $k$ even and $h+1=(k+1)/2$ if $k$ is odd, then 
   \eqref{eqn:limit-of-the-branch}  holds.
   \end{proof}

  \begin{proof}[Proof of Theorem \ref{thm:global-bifurvation-radial}]
  If $\eps\in]0,1[$ we set:
  \begin{equation*}
   \mathcal{O}_\eps:=\set{(w,\lm)\in E\st \eps<\lm<\eps^{-1},
   1+\sqrt{\lm}w(\rho)>\eps,w(\rho)<\eps^{-1}\ \forall\rho\in[0,R]}
  \end{equation*}
  Clearly $\mathcal{O}_\eps$ is an open set with 
  $\mathcal{O}_\eps\subset\mathcal{W}$. 
  Moreover $(\mu_{k/2},0)\in\mathcal{O}_\eps$ if $\eps$
  is sufficently small. Define $\tilde h_{\lm,\eps}$ as in
  \eqref{eqn:def-tilde-h-1} with $s_0=\eps$ and let 
  $\tilde h_{\lm}(s):=\tilde h_1(\sqrt{\lm}s)$. Using \cite{RabinowitzSturmLiouville70}
  we get there that there exists a pair $(w_\eps,\lm_\eps)$
  in $\partial\mathcal{O}_\eps$,
  with $w_\eps$ having  $k$ nodal points, which
  solves Problem \eqref{eqn:radial-equation} with 
  $\tilde h_{\eps,\lm}:=\tilde h_\eps(\lm,\cdot)$ instead of $h_\lm$. 
  Since $(w,\lm)\in\partial\mathcal{O}_\eps\Rightarrow \tilde h_\eps(w,\lm)=h_\lm(w)$, we get that 
  $(w_\eps,\lm_\eps)\in\mathcal{S}_k^+$.
  For $\eps$ small we have $\eps<\underline{\lm}_k\leq\overline{\lm}_k<\eps^{-1}$ so we get $w_\eps\in\partial\set{1+\sqrt{\lm_\eps}w>\eps,w<\eps^{-1}}$ i.e. there exists a point $\rho_\eps\in[0,R]$ such that 
  \begin{equation*}
  \mbox{either}\quad  1+\sqrt{\lm_\eps}w_\eps(\rho_\eps)=\eps
  \qquad\quad
  \mbox{or}\quad  w_\eps(\rho_\eps)=\eps^{-1}.
  \end{equation*}
  We can find a sequence $\eps_n\to0$ such that the corresponding $(w_n,\lm_n):= (w_{\eps_n},\lm_{\eps_n})$
  verify one of the above properties for all $n\in\nat$.
  If the first one holds for all $n$, then $(w_n,\lm_n)$
  verifies \thetag{b} of Lemma
  \eqref{lma:main-lemma-sequences}; in the second case 
  $(w_n,\lm_n)$ verifies \thetag{a} of Lemma
  \eqref{lma:main-lemma-sequences}. Then by Lemma 
  \eqref{lma:main-lemma-sequences} $\|w_n\|_\infty\to\infty$
  and \eqref{eqn:limit-of-the-branch} holds. 
  This proves the theorem.
  \end{proof}

   \begin{rmk}
    As a consequence of Theorem \eqref{thm:global-bifurvation-radial}
    we get that for any $h\geq1$ integer and any $\lm$ strictly between $\lm_h$ and $\lm_{2h}/2$ there exists $u$ such that $(u,\lm)$ solves Problem \eqref{eqn:singular-neumann-problem}.
    The same is true for all $\lm$ strictly between $\nu_h$ and $\lm_{2h-1}/2$
   \end{rmk}
   
   \begin{rmk}
    The above proof  fails if we follow the bifurcation branch
    $(w_\rho,\lm_\rho)$ with $w_\rho(0)<0$. In this case it seems possible that the branch tends to a point $(\tilde\lm,\tilde w)$
    where $\sqrt{\tilde\lm}\tilde w(0)=-1$ 
    (but $\sqrt{\tilde\lm}\tilde w(0)>-1$ for $\rho>0$). This phenomenon, if true, would be worth studying.
   \end{rmk}
   
   \begin{rmk}\label{eqn:rmk-no-dirichlet-radial-soln}
    The computations of this section show that, if $\Omega$ is the ball, then there are no solutions for the Dirichlet problem. It is
    indeed impossible to construct a (nontrivial) solution $(w,\lm)$ for
    \eqref{eqn:radial-equation} with   $w(R)=0$,
   \end{rmk}


\end{document}

%% file: defs.tex
\newcommand{\real}{\ensuremath{\mathbb{R}}}

\newcommand{\nat}{\ensuremath{\mathbb{N}}}

\newcommand{\CC}{\mathcal{C}}
\newcommand{\SSS}{\mathcal{S}}

\newcommand{\X}[1][{}]{\ensuremath{\mathbb{X}}}

\newcommand{\ph}{\varphi{}}
\newcommand{\eps}{\varepsilon{}}
\newcommand{\sm}{{\sigma}}
\newcommand{\lm}{{\lambda}}

\newcommand{\set}[1]{\ensuremath{ \left\{ #1 \right\} }}

\newcommand{\st}{{\,:\,}}
\newcommand{\scal}[2]{\ensuremath{\left\langle#1,#2\right\rangle}}

\newcommand{\norm}[1]{\ensuremath{\left\lVert #1 \right\rVert}}

\newcommand{\uno}{\mathbb{1}}

\newcommand{\Hone}{\ensuremath{W^{1,2}_0(\Omega)}}

\newcommand{\Hon}{\ensuremath{W^{1,2}(\Omega)}}
\newcommand{\Ltwo}{\ensuremath{L^2(\Omega)}}

\newcommand{\Linfty}{\ensuremath{L^\infty(\Omega)}}
\newcommand{\Lpi}[1][p]{\ensuremath{L^{#1}(\Omega)}}

\swapnumbers
\theoremstyle{plain}
\newtheorem{thm}{Theorem}[section]
\newtheorem{lma}[thm]{Lemma}

\theoremstyle{definition}

\theoremstyle{remark}
\newtheorem{rmk}[thm]{Remark}

\DeclareMathOperator{\spanned}{span}
